\documentclass[12pt]{article}

\usepackage{amsmath,amsfonts,amssymb,amsthm, mathrsfs}
\usepackage{fullpage}
\usepackage{stmaryrd}
\usepackage{hyperref}
\usepackage{tikz}
\usetikzlibrary{positioning,decorations,arrows,patterns,shapes,backgrounds}
\usepackage{color}

\newcommand{\blue}{}
\newcommand{\red}{}

\newcommand{\blu}{}
\newcommand{\rosso}{}

\newcommand{\orange}{}

\newcommand{\equlaw}{\stackrel{(d)}{=}}

  \newcommand{\coef}{B_{n}(f)}
  \newcommand{\coeff}{B'_{n}(f)}

\theoremstyle{definition}
\newtheorem{definition}{Definition}
\newtheorem{remark}[definition]{Remark}
\newtheorem{example}[definition]{Example}

\newtheoremstyle{mytheorem}{0.5cm}{0.2cm}{\slshape}{ }{\bfseries}{.}{ }{}
\theoremstyle{mytheorem}
\newtheorem{theorem}[definition]{Theorem}
\newtheorem{lemma}[definition]{Lemma}
\newtheorem{proposition}[definition]{Proposition}

\newtheorem{corollary}[definition]{Corollary}

\newtheorem{remarks}{Remarks}

\newcommand{\one}{\mathbf{1}}
\renewcommand{\P}{\mathbf{P}}

\newcommand{\E}{\mathbf{E}}

\newcommand{\Vol}{\text{Vol}}

\newcommand{\Per}{\text{Per}}

\newcommand{\var}{{\bf Var}}

\newlength{\querylen}
\usepackage{fancybox}

\numberwithin{equation}{section}
\numberwithin{definition}{section}

\title{}

\begin{document}

\begin{center}
\large {\bf NEW KOLMOGOROV BOUNDS FOR FUNCTIONALS }\\
\large{\bf OF BINOMIAL POINT PROCESSES\footnote{This research has been supported by the grant F1R-MTH-PUL-12PAMP (PAMPAS) at Luxembourg University}}

 \medskip

\medskip

\medskip

\normalsize
by Rapha\" el Lachi\`eze-Rey\footnote{Laboratoire MAP5 
UniversitŽ\'e Paris Descartes, Sorbonne Paris Cit\'e, Paris. Email: raphael.lachieze-rey@parisdescartes.fr.} and Giovanni Peccati\footnote{Unit\'e de Recherche en Math\'ematiques, Universit\'e du Luxembourg, Luxembourg. Email: giovanni.peccati@gmail.com. }

\end{center}

{\small \noindent {\bf Abstract}: We obtain explicit Berry-Esseen bounds in the Kolmogorov distance for the normal approximation of non-linear functionals of vectors of independent random variables. Our results are based on the use of Stein's method and of random difference operators, and generalise the bounds recently obtained by Chatterjee (2008), concerning normal approximations in the Wasserstein distance. In order to obtain lower bounds for variances, we also revisit the classical Hoeffding decompositions, for which we provide a new proof and a new representation. Several applications are discussed in detail: in particular, new Berry-Esseen bounds are obtained for set approximations with random tessellations, as well as for functionals of covering processes.  
\\

\section{Introduction}

\subsection{Overview}

Let $X = (X_1,...,X_n)$ be a collection of independent random variables, defined on some probability space $(\Omega, \mathscr{F}, \P)$ and taking values in some Polish space $(E,\mathscr{E})$; let $f :E^n \to \mathbb{R}$ be a measurable function such that $f(X)$ is square-integrable. The aim of the present paper is to deduce a new class of explicit upper bounds for the {\it Kolmogorov distance} $d_{K}(f(X), N)$, between the distribution of $f(X)$ and that of a Gaussian random variable $N\sim \mathscr{N}(m, \sigma^2)$ such that $m = {\bf E}f(X)$ and $\sigma^2 = {\bf Var} f(X)$. Recall that $d_{K}(f(X), N)$ is defined as:
$$
d_{K}(f(X), N) = \sup_{t\in \mathbb{R}} \left| {\bf P}[f(X)\leqslant t] - {\bf P}[N\leqslant t]  \right|.
$$
The problem of obtaining explicit estimates on the distance between the distributions of $f(X)$ and $N$ has been recently dealt with in the paper $\cite{Cha08}$, where the author was able to apply a standard version of Stein's method (see e.g. \cite{CGS2011}) in order to deduce effective upper bounds on the {\it Wasserstein distance}
$$
d_{W}(f(X), N) = \sup_{h} \left| {\bf E}[h (f(X)) ] - {\bf E}[h(N)]  \right|,
$$where the supremum runs over $1$-Lipschitz functions,
by using a class of difference operators that we shall explicitly describe in Section \ref{ss:difference} below (see e.g. \cite{Cha13, hi, Nol} for some relevant applications of these bounds). 

\medskip

It is a well known fact that upper bounds on $d_{W}(f(X), N)$ also yield a (typically suboptimal) bound on $d_{K}(f(X), N)$ via the standard relation $d_{K}(f(X), N)\leqslant 2 \sqrt{d_{W}(f(X), N)}$. The challenge we are setting ourselves in the present paper is to deduce upper bounds on $d_{K}(f(X), N)$ that are {\it {potentially} of the same order} as the bounds on $d_{W}(f(X), N)$ that can be deduced from \cite{Cha08}. {\blu Our main abstract findings appear in the statement of Theorem \ref{thm:SchulteSchatterjeebound} below}. In order {\blu to prove our main bounds}, we shall exploit some novel estimates on the solution of the Stein's equations associated with the Kolmogorov distance, that are strongly inspired by computations developed in \cite{EicTha14, Sch12} in the framework of normal approximations for functionals of Poisson random measures.

\medskip

Another important contribution of the present work (ses Section \ref{ss:hoeffding}) is a novel representation (in terms of difference operators) of the kernels determining the {\it Hoeffding decomposition} (see e.g. \cite{hoeffding, peccati, vitale}, as well as \cite[Chapter 5]{serfling}) of a random variable of the type $f(X)$. This new representation is put into use for deducing effective lower bounds on ${\bf Var} f(X)$.

\medskip

As demonstrated in the sections to follow, we are mainly interested by geometric applications and, in particular, by the normal approximation of geometric functionals whose dependency structure can be {assessed} by using { second order difference operators}. One of the applications developed in detail in Section \ref{sec:voronoi} is that of {\it Voronoi set approximations}, where a given set $K$ is estimated by the union of Voronoi cells. Remarkably, our bounds allow one to deduce normal approximation bounds for the volume approximation of sets $K$ having a highly non-regular boundary. The present paper is associated with the work \cite{LacVeg}, where it is proved that, for a large class of sets with self-similar boundary of dimension $s>d-1$, the variance of the volume approximation is asymptotically {\blu of the same order as} $n^{-2+s/d}$ and the Kolmogorov distance between the volume approximation and the normal law is smaller than {\blu some multiple of} $n^{- s/2d }$ multiplied by a logarithmic term. It turns out that the crucial feature for a set to be well behaved with respect to Voronoi approximation is its density at the boundary, which is mathematically independent of its fractal dimension (see \cite{LacVeg} for an in-depth discussion of these phenomena). For illustrative purposes, we will also present an application of our methods to covering processes (re-obtaining the results of \cite{GolPen10} in a slightly more general framework, see Section \ref{sec:covering} below), as well as to some models already studied in \cite{Cha08} and \cite{Nol}.

\smallskip

{ In the recent reference \cite{gn}, Gloria and Nolen have effectively used Theorem \ref{thm:SchulteSchatterjeebound} below for deducing Berry-Esseen bounds in the Kolmogorov distance for the effective conductance on the discrete torus.}

\subsection{Plan}

Section 2 contains our main results concerning decompositions of random variables. Section 3 deals with some estimates associated with Stein's method, and Section 4 contains our main abstract findings. Section 5 focusses on estimates based on   {  second order difference operators}. Finally, several applications are developed in Section 6.

\medskip

From now on, every random object is defined on an adequate common probability space $(\Omega, \mathscr{F}, \P)$, with $\E$ denoting expectation with respect to $\P$.

\section{Decomposing random variables}
 
\subsection{\blue Some difference operators}\label{ss:difference}

Let $(E,\mathscr{E})$ be a Polish space endowed with its Borel $\sigma$-field.  Given two vectors $y =(y_1,...,y_n)\in E^n$ and $y' = (y'_1,...,y'_n)\in E^n$, for every $C\subseteq [n]{\blue :=\{1,...,n\}}$ and every measurable function $f : E^n \to \mathbb{R}$, we {\blue denote by} $f^C(y,y')$ the quantity that is obtained from $f(y)$ by replacing $y_i$ with $y'_i$ whenever $i\in C$. For instance, if $n=4$ and $C= \{1,4\}$, then
\[
f^C(y,y') = f(y'_1,y_2,y_3,y'_4)
\]
and 
\[
f^C(y',y) = f(y_1,y'_2,y'_3,y_4).
\]
Given $C\subseteq [n]$, we introduce the operator
\[
\Delta_C f(y,y') = f(y) -f^C(y,y').
\]
When $C=\{j\}$ (to simplify the notation), we shall often write $f^{\{j\}} = f^j$ and $\Delta_{\{j\}}= \Delta_j$, for $j=1,...,n$, in such a way that
\[
\Delta_{\{j\}}f(y,y') = \Delta_j f(y,y') =f(y) - f^j(y,y')=  f(y) - f(y_1,...,y_{j-1},y'_j, y_{j+1},...,y_n), 
\]
and 
\[
\Delta_{\{j\}}f(y',y) = \Delta_j f(y',y) = f(y') - f^{j}(y',y)= f(y') - f(y'_1,...,y'_{j-1},y_j, y'_{j+1},...,y'_n).
\]
{\blue We can canonically iterate} the operator $\Delta _{j}$ as follows: {\blue for every $k\geq 2$ and every choice of distinct indices $1\leqslant i_1< \cdots<i_k \leqslant n$, the quantity $\Delta_{i_1}\cdots \Delta_{i_k}f(y,y')$, is defined as
\[
\Delta_{i_1}\cdots \Delta_{i_{k-1}}f(y,y') - (\Delta_{i_1}\cdots \Delta_{i_{k-1}}f(y,y') )_{i_k},
\]
where $(\Delta_{i_1}\cdots \Delta_{i_{k-1}}f(y,y') )_{i_k}$ is obtained by replacing $y_{i_k}$ with $y'_{i_k}$ inside the argument of $$\Delta_{i_1}\cdots \Delta_{i_{k-1}}f(y,y').$$ Note that {\blue the operator $\Delta_{i_1}\cdots \Delta_{i_k}$} defined in this way is invariant with respect to permutations of the indices $i_1,...,i_k$.} For instance, if $n=2$,
\begin{eqnarray*}
\Delta_1\Delta_2 f(y,y') &=&\Delta_2\Delta_1 f(y,y') \\
&=& f(y'_1,y'_2)-f(y'_1,y_2) - f(y_1,y'_2)+f(y_1,y_2).
\end{eqnarray*}
{\blue The notation introduced above} also extends to random variables: if $X = (X_1,...,X_n)$ and $X' = (X'_1,...,X'_n)$ are two random vectors with values in $E^n$, then we write
\[
\Delta_C f(X,X') := f(X) - f^C(X,X'),\quad C\subseteq[n],
\]
and define $\Delta_{i_1}\cdots \Delta_{i_k}f(X,X')$, $1\leqslant i_1< \cdots<i_k \leqslant n$, exactly as above. The definitions of $\Delta_C f(X',X)$ and $\Delta_{i_1}\cdots \Delta_{i_k}f(X',X)$ are given analogously. Now assume that $E[| f(X)|]<\infty$. Our aim in this section is to discuss two representations of the quantity $f(X) - E[f(X)]$, that are based on the use of the difference operators $\Delta_j$. The first one is a reformulation of the classical {\it Hoeffding decomposition} for functions of independent random variables (see e.g. \cite{hoeffding, peccati, vitale}, as well as \cite[Chapter 5]{serfling}). The second one comes from \cite{Cha08} (see also \cite[Chapter 7]{Cha13}) and will play an important role in the derivation of our main estimates.

\subsection{A new look at Hoeffding decompositions}\label{ss:hoeffding}

{\blue Throughout this section}, for every fixed integer $n\geq 1$ we write $X = (X_1,...,X_n)$ to indicate a vector of independent random variables with values in {\blue the Polish space} $E$, and let $X' = (X'_1,...,X'_n)$ be an independent copy of $X$. If $f: E^n \to \mathbb{R}$ is a measurable function such that $\E[f(X)^2]<\infty$, then the classical theory of Hoeffding decompositions for functions of independent random variables (see e.g. \cite{KR, vitale}) implies that $f(X)$ admits a unique decomposition of the type
\begin{equation}\label{e:ch} 
f(X) = \E[f(X)] +\sum_{k=1}^n \sum_{1\leqslant i_1<\cdots <i_k\leqslant n} \varphi_{i_1,...,i_k} (X_{i_1},...,X_{i_k}),
\end{equation}
where the square-integrable kernels $\varphi_{i_1,...,i_k}$ verify the degeneracy condition
\[
\E[ \varphi_{i_1,...,i_k} (X_{i_1},...,X_{i_k})\, |\, X_{j_1},...,X_{j_a}]=0,
\]
for any strict subset $\{j_1,...,j_a\}$ of $\{i_1,...,i_k\}$. The derivation of \eqref{e:ch} is customarily based on some implicit recursive application of the inclusion-exclusion principle, and the kernels $\varphi_{i_1,...,i_k}$ can be represented as linear combinations of conditional expectations. As abundantly illustrated in the above-mentioned references, a representation such as \eqref{e:ch} is extremely useful for analysing the variance of a wide range of random variables (in particular, $U$-statistics). Our aim in the present section is to  point out a very compact way of writing the decomposition \eqref{e:ch}, that is based on the use of the operators $\Delta_j$ introduced above. Albeit not surprising, such an approach towards Hoeffding decompositions seems to be new {\blue and of independent interest}, and will be quite useful in the present paper for explicitly deriving lower bounds on variances. Our starting point is the following statement, where we make use of the notation introduced in Section \ref{ss:difference}.

\begin{lemma}\label{l:1} For every $f : {\blue E}^n \to \mathbb{R}$
\begin{equation}\label{e:a}
f(y) - f(y') =  \sum_{k=1}^n \sum_{1\leqslant i_1<\cdots < i_k\leqslant n} (-1)^k \Delta_{i_1}\cdots \Delta_{i_k}f(y',y).
\end{equation}
\end{lemma}
\begin{proof} The key observation is that, for every $k\geq 1$ and every $B = \{i_1,...,i_k\}$,
\[
\Delta_{i_1}\cdots \Delta_{i_k}f(y',y) = \sum_{A\subseteq B} (-1)^{|A|} f^A(y',y),
\]
a relation that can be easily proved by recursion. By virtue of this fact, one can now rewrite the right-hand side of \eqref{e:a} as 
\begin{equation}\label{e:b}
\sum_{A\subseteq [n]} \psi(A)\times  Z(A),
\end{equation}
where $\psi(A) :=  f^A(y',y)$ and $Z(A) := \sum_{B:B\neq \emptyset, A\subseteq B} (-1)^{|B\backslash A|}$. Standard combinatorial considerations yield that $Z([n]) = 1$, $Z(\emptyset) = -1$ and $Z(A) = 0$, for every non-empty strict subset of $[n]$. This implies that \eqref{e:b} is indeed equal to $\psi([n]) - \psi(\emptyset)$, and the desired conclusion follows at once.\end{proof}

\medskip

Now fix an integer $n$, as well as {\blue $n$-dimensional} vectors $X$ and $X'$ as above {\blue (in particular, $X'$ is an independent copy of $X$)}: the following statement provides an alternate description of the Hoeffding decomposition of $f(X)$ in terms of the difference operators defined above.

\begin{theorem}[Hoeffding decompositions]\label{t:hd}
Let $f : E^n \to \mathbb{R}$ be such that $E[f(X)^2]<\infty$. One has the following representation for $f(X)$:
\begin{equation}\label{e:newh}
f(X) = \E[f(X)] + \sum_{k=1}^n \sum_{1\leqslant i_1<\cdots< i_k \leqslant n} (-1)^k \E\left[ \Delta_{i_1}\cdots \Delta_{i_k} f(X',X) | X\right].
\end{equation}
Formula \eqref{e:newh} coincides with the Hoeffding decomposition \eqref{e:ch}  of $f(X)$: in particular, one has that, for any choice of $i_1,...,i_k$, $\E\left[ \Delta_{i_1}\cdots \Delta_{i_k} f(X',X) | X\right] = \varphi_{i_1,...,i_k} (X_{i_1},...,X_{i_k})$, and consequently
\begin{equation}\label{e:cov}
\E\Big\{ \E\left[ \Delta_{i_1}\cdots \Delta_{i_k} f(X',X) | X\right]\times \E\left[ \Delta_{j_1}\cdots \Delta_{j_l} f(X',X) | X\right] \Big\} = 0,
\end{equation}
whenever $\{i_1,...,i_k\} \neq \{j_1,...,j_l\}$.
\end{theorem}
\begin{proof}
By Lemma \ref{l:1},
\[
f(X) = f(X') + \sum_{k=1}^n \sum_{1\leqslant i_1<\cdots i_k \leqslant n} (-1)^k  \Delta_{i_1}\cdots \Delta_{i_k} f(X',X),
\]
and \eqref{e:newh} follows at once by taking conditional expectations {\blu with respect to $X$} on both sides. To prove \eqref{e:cov}, it suffices to show the following stronger result: for every $1\leqslant i_1<\ldots < i_k\leqslant n$ (all $k$ indices different),
\[
\E\left[ \Delta_{i_1}\cdots \Delta_{i_k} f(X',X) | X_{i_1},\ldots, X_{i_{k-1}} \right] = 0.
\]
This {\blue is a consequence of} the following fact: the random variable $ \Delta_{i_1}\cdots \Delta_{i_{k-1}} f(X',X)$ is a function of $X_{i_1},\ldots, X_{i_{k-1}}$ and of $X'$.  By independence, it follows that 
\[
\E\left[ \Delta_{i_1}\cdots \Delta_{i_{k-1} } f(X',X) | X_{i_1},\ldots, X_{i_{k-1}} \right] = \E\left[ (\Delta_{i_1}\cdots \Delta_{i_{k-1} } f(X',X))_{i_k} | X_{i_1},\ldots, X_{i_{k-1}} \right] 
\]
where the random variable $(\Delta_{i_1}\cdots \Delta_{i_{k-1} } f(X',X))_{i_k}$ has been obtained from $\Delta_{i_1}\cdots \Delta_{i_{k-1} } f(X',X)$ by replacing $X'_{i_k}$ with $X_{i_k}$. Since (as already observed) 
\[
\Delta_{i_k}\Delta_{i_1}\cdots \Delta_{i_{k-1} } f(X',X) = \Delta_{i_1}\cdots \Delta_{i_{k} } f(X',X), 
\]
{\blue we deduce immediately the desired conclusion.}
\end{proof}

{\blue The next statement is a direct consequence of \eqref{e:newh}--\eqref{e:cov}.}

\begin{corollary} Let $f(X)$ be as in the statement of Theorem \ref{t:hd}. Then, the variance of $f(X)$ can be expanded as follows:
\begin{equation}\label{e:v}
\var(f(X)) = \sum_{k=1}^n \sum_{1\leqslant i_1<\cdots< i_k \leqslant n}\E\left[\left(  \E\left[ \Delta_{i_1}\cdots \Delta_{i_k} f(X',X) | X\right] \right)^2\right].
\end{equation}
\end{corollary}

As a first application of \eqref{e:v}, we present a useful lower bound for variances.

\begin{corollary}
\label{lem:var-lower-bound} {\blue Let $f(X)$ be as in the statement of Theorem \ref{t:hd}. Then, one has the lower bound
\[
\var(f(X)) \geq  \sum_{i=1}^n \E\left[\left(  \E\left[ \Delta_{i}f(X',X) | X\right] \right)^{2}\right]
\]
In particular}, if $X =(X_{1},\dots ,X_{n})$ is a collection of $n$ i.i.d. random variables with common distribution equal to $\mu$, and $f:E^n \to \mathbb{R}$ is a symmetric mapping such that $E[f(X)^2]<\infty$, then 
\begin{align*}
\var(f(X) )\geq n\int_{E}\left( \E [f (X)-f(x,X_{2},\dots ,X_{n})] \right)^{2}\, \mu(dx).
\end{align*}
\end{corollary}

\begin{remark}
The estimates in {\blue Corollary} \ref{lem:var-lower-bound} should be compared with the classical {\it Efron-Stein inequality} (see e.g. \cite[Chapter 3]{BLMbook}), {\blue stating that
\[
\var(f(X)) \leqslant \frac12  \sum_{i=1}^n \E\left[\Delta_{i}f(X,X')^{2}\right],
\]
which, in the case where the $X_i$ are i.i.d. and $f$ is symmetric, becomes}
\begin{align*}
\var(f(X) )\leqslant \frac{n}{2}\int_{E}  \E [(f (X)-f(x,X_{2},\dots ,X_{n}))^2] \, \mu(dx).
\end{align*}
For instance, if $f(X)=X_{1}+\dots +X_{n}$ {\blue is a sum of real-valued independent and square-integrable random variables, then the Efron-Stein upper bounds coincides with the lower bound in Corollary \ref{lem:var-lower-bound}, that is:
$$
\sum_{i=1}^n \E\left[\left(  \E\left[ \Delta_{i}f(X',X) | X\right] \right)^{2}\right]=\frac12  \sum_{i=1}^n \E\left[\Delta_{i}f(X,X')^{2}\right] = \sum_{i=1}^n {\bf Var}(X_i).
$$}{\blue Heuristically}, in the general case where the $X_i$ are i.i.d. and $f$ is symmetric, it seems that, {\blue in order for the Efron-Stein upper bound and  the lower bound of {\blue Corollary} \ref{lem:var-lower-bound} to have the same magnitude, it is necessary that the functional $f(X)$ is not {\it homogeneous}, meaning that the law of $f (X)-f(x,X_{2},\dots ,X_{n})$ depends on $x$.  Examples of such a behaviour will be described in Section \ref{sec:voronoi}, where we will deal with Voronoi approximations.}

%
\end{remark}

\subsection{Another subset-based interpolation}\label{ss:another}

Let $n\geq 1$, let $f: E^n \to \mathbb{R}$, and let $y,y'\in E^n$. In \cite{Cha08}, the following formula is pointed out: \begin{equation}\label{e:c}
f(y) - f(y') = \sum_{A\subsetneq [n]} \frac{1}{\binom{n}{|A|} (n-|A|)}\sum_{j\notin A} \Delta_j f(y^A,y'),
\end{equation}
where the vector $y^A$ has been obtained from $y$ by replacing $y_i$ with $y'_i$ whenever $i\in A$, {\blue in such way that, with our notation}, $\Delta_j f(y^A,y') =f(y^A) - f(y^{A\cup\{j\}}) = f^{A}(y,y') - f^{A\cup\{j\}}(y,y')$.  

Now consider a vector $X = (X_1,...,X_n)$, with independent components and with values in $E^n$, and let $X'$ be an independent copy of $X$. For every $A\subseteq [n]$, we define $X^A = (X_1^A,..., X_n ^A)$ according to the above convention, that is:
\begin{align*}
 X_{i}^{A}=
\begin{cases}X_{i}$ if $i\notin A\\
X'_{i}$ otherwise$.
\end{cases}
\end{align*}

The following statement is a direct consequence of \eqref{e:c}.

\begin{proposition}[See \cite{Cha08}, Lemma 2.3] For every $f,g : A^n \to \mathbb{R}$ such that $E[f(X)^2],E[g(X)^2]<\infty$,
\begin{equation}\label{e:d}
{\bf Cov} (f(X), g(X)) = {\blue \frac12}\sum_{A\subsetneq [n]} \frac{1}{\binom{n}{|A|} (n-|A|)}\sum_{j\notin A}\E[\Delta_j g(X,X')\Delta_j f(X^A,X')].
\end{equation}

\end{proposition}
{\blue 

To simplify the notation, we shall sometimes write $$\frac{1}{ \binom{n}{ | A | }(n- | A | ) }:=\kappa _{n,A}.$$ Observe that, for every $j$, $\sum_{A\subsetneq [n]:j\notin A}\kappa _{n,A} =1$

\begin{remark} As demonstrated in \cite[Lemmas 7.8-7.10]{Cha13}, the identity \eqref{e:d} can also be used to deduce effective lower bounds on variances. Such lower bounds seem to have a different nature from the ones that can be proved by means of Hoeffding decompositions.
\end{remark}
}

{\blue 

\section{\blu Stein's method and a new approximate Taylor expansion}\label{s:stein}

Let $U$ and $V$ be two real-valued random variables. The {\it Kolmogorov distance} between the distributions of $U$ and $V$ is given by
\begin{align*}
d_{K}(U,V)=\sup_{t\in \mathbb{R}} | \P(U\leqslant t)-\P(V\leqslant t) | .
\end{align*}

As anticipated in the Introduction, our aim in this paper is to provide upper bounds for quantities of the type $d_K(W,N)$, where $W = f(X)$ and $N$ is a standard Gaussian random variable, that are based on the use of Stein's method. The following statement gathers together some classical facts concerning Stein's equations and their solutions (see Points (a)--(e) below), together with a new {\blu important approximate Taylor expansion for solutions of Stein's equations, that we partially} extrapolated from reference \cite{EicTha14} (see Point (f) below), generalising previous findings from \cite{Sch12}; see also \cite[Theorem 2]{BouPecSurvey}.

\begin{proposition}\label{p:stein} Let $N\sim \mathcal{N}(0,1)$ be a centred Gaussian random variable with variance 1 and, for every $t\in\mathbb{R}$, consider the Stein's equation
\begin{align}
\label{eq:stein-equation}
g'(w)-wg(w)=\one_{w\leqslant t}-\P(N\leqslant t),
\end{align}
where $w\in \mathbb{R}$. Then, for every real $t$, there exists a function $g_t : \mathbb{R}\to \mathbb{R} : w\mapsto g_t(w)$ with the following properties:
\begin{itemize}

\item[\rm (a)] $g_t$ is continuous at every point $w\in \mathbb{R}$, and infinitely differentiable at every $w\neq t$;

\item[\rm (b)] $g_t$ satisfies the relation \eqref{eq:stein-equation}, for every $w\neq t$;

\item[\rm (c)] $0<g_{t}\leqslant c:= \frac{\sqrt{2\pi }}{4}$;

\item[\rm (d)] for every $u,v,w\in \mathbb{R}$,
\begin{equation}\label{e:zumba}
\vert (w+u)g_t(w+u) - (w+v)g_t(w+v)\vert \leqslant \left( \vert w \vert + \frac{\sqrt{2\pi}}{4}\right) \left(\vert u \vert + \vert v \vert \right);
\end{equation}

\item[\rm (e)] adopting the convention 
\begin{equation}\label{e:convention}
g'_t(t) : = tg_t(t)+1-\P(N\leqslant t), \quad 
\end{equation}
one has that $|g'_t(w)|\leqslant 1$, for every real $w$.

\item[\rm (f)] using again the convention \eqref{e:convention}, for all $w,h\in \mathbb{R}$ one has that
\begin{align}
\label{eq:stein}
 | g_{t}(w+h)-g_{t}(w)-g_{t}'(w)h | &\leqslant \frac{ |  h|^{2} }{2}\left(  | w| +\frac{\sqrt{2\pi }}{4} \right)\\
 &\quad\quad\quad + | h | (\one_{[w, w+h)}(t) + \one_{[w+h, w)}(t))\notag \\
 &= \frac{ |  h|^{2} }{2}\left(  | w| +\frac{\sqrt{2\pi }}{4} \right)\label{eq:stein2} \\
 &\quad\quad\quad+ h\left(\one_{ \left[ w , w + h \right)}(t) - \one_{\left[  w+h , w\right)}(t) \right).\notag
\end{align}

\end{itemize}
\end{proposition}

\begin{proof} The proofs of Points (a)--(e) are classical, and can be found e.g. in \cite[Lemma 2.3]{CGS2011}. We will prove (f) by following the same line of reasoning adopted in \cite[Proof of Theorem 3.1]{EicTha14}. Fix $t\in \mathbb{R}$, recall the convention \eqref{e:convention} and observe that, for every $w,h \in \mathbb{R}$, we can write
\begin{equation*}
g_t(w+h)-g_t(w) - hg'_t(w) = \int_{0}^{h}\left(g_t'(w+u) - g'(w) \right) du.
\end{equation*}  
Since $g_t$ solves the Stein's equation \eqref{eq:stein-equation} for every real $w$, we have that, for all $w,h \in \mathbb{R}$,
\begin{eqnarray*}
&& g_t(w+h)-g_t(w) - hg'_t(w) \\
&& = \int_{0}^{h}\left((w+u)g_t(w+u) - wg_t(w) \right) du + \int_{0}^{h}\left(\one_{\left\lbrace w + u \leqslant t\right\rbrace } - \one_{\left\lbrace w \leqslant t\right\rbrace } \right) du := I_1 + I_2.
\end{eqnarray*} 
It follows that, by the triangle inequality,
\begin{equation}
\label{boundingthediff}
\left| g_t(w+h)-g_t(w) - hg'_t(x)\right| \leqslant \vert I_1 \vert + \vert I_2 \vert.
\end{equation} 
Using \eqref{e:zumba}, we have 
\begin{equation}
\label{boundI1}
\vert I_1 \vert \leqslant \int_{0}^{h}\left(|w| +\frac{\sqrt{2\pi}}{4}\right)\vert u \vert du = \frac{h^2}{2}\left(|w| +\frac{\sqrt{2\pi}}{4}\right).
\end{equation} 
Furthermore, observe that 
\begin{eqnarray*}
\vert I_2 \vert &=& \one_{\left\lbrace h <0\right\rbrace }\left| \int_{0}^{h}\left(\one_{\left\lbrace w + u \leqslant t\right\rbrace } - \one_{\left\lbrace w \leqslant t\right\rbrace } \right) du \right| + \one_{\left\lbrace h \geq 0\right\rbrace }\left| \int_{0}^{h}\left(\one_{\left\lbrace w + u \leqslant t\right\rbrace } - \one_{\left\lbrace w \leqslant t\right\rbrace } \right) du \right| \\
&=& \one_{\left\lbrace h <0\right\rbrace }\left| -\int_{h}^{0}\one_{\left\lbrace w+u \leqslant t < w\right\rbrace }du \right| + \one_{\left\lbrace h \geq 0\right\rbrace }\left| -\int_{0}^{h}\one_{\left\lbrace w \leqslant  t <w + u \right\rbrace } du\right| \\
&=& \one_{\left\lbrace h <0\right\rbrace }\int_{h}^{0}\one_{\left\lbrace w+u \leqslant t < w\right\rbrace }du  + \one_{\left\lbrace h \geq 0\right\rbrace }\int_{0}^{h}\one_{\left\lbrace w \leqslant  t < w + u \right\rbrace } du.
\end{eqnarray*}
Bounding $u$ by $h$ in both integrals provides the following upper bound:
\begin{eqnarray}
\vert I_2 \vert & \leqslant & \one_{\left\lbrace h <0\right\rbrace }(-h)\one_{\left[  w+h , w\right)}(t)  + \one_{\left\lbrace h \geq 0\right\rbrace }h\one_{ \left[ w , w + h \right)}(t)\nonumber \\
& \leqslant & h\left(\one_{ \left[w , w + h \right)}(t) - \one_{\left[  w+h , w\right)}(t) \right)\label{boundI2} = \vert h \vert \left(\one_{ \left[ w , w + h \right)}(t) + \one_{\left[  w+h , w\right)}(t) \right).
\end{eqnarray}
Applying the estimates \eqref{boundI1} and \eqref{boundI2} to \eqref{boundingthediff} concludes the proof.
\end{proof}

An immediate consequence of Proposition \ref{p:stein} is that for $N\sim \mathcal{N}(0,1)$ and for every real-valued random variable $W$, one has that 
\begin{align}\label{e:upkol}
d_{K}(W,N)  =\sup_{t\in \mathbb{R}} |  \E g_{t}'(W)-Wg_{t}(W)|
\end{align}
{(observe in particular that convention \eqref{e:convention} defines unambiguously the quantity $g'_t(x)$ for every $t,x\in \mathbb{R}$) }.

}

\section{{ New} Berry-Esseen bounds { in the Kolmogorov\\ distance} }
 {\blue 
Let $n\geq 1$ be an integer, and consider a a vector $X = (X_1,...,X_n)$ of independent random variables with values in the Polish space $E$. Let $X'=(X_{1}',\dots ,X')$ be an independent copy of $X$. Consider a function $f : E^n \to \mathbb{R}$ such that $W:= f(X)$ is a centred and square-integrable random variable. We shall adopt the same notation introduced in Sections \ref{ss:difference}, \ref{ss:hoeffding}, \ref{ss:another} and \ref{s:stein}. For every $A\subsetneq [n]$, we write
\begin{align*}
T_{A}=\sum_{j\notin A}\Delta _{j}f(X,X')\Delta _{j}f(X^A,X')\\
T_{A}'=\sum_{j\notin A}\Delta _{j}f(X,X') | \Delta _{j}f(X^A,X') | 
\end{align*} 
and
\begin{align*}
T=\frac{1}{2}\sum_{A\subsetneq [n]}\kappa _{n,A}{T_{A}}, \\
T'=\frac{1}{2}\sum_{A\subsetneq [n]}\kappa _{n,A}{T_{A}'}. 
\end{align*} 
Observe that each $T_{A}'$ is a sum of symmetric random variables in such way that $0 = \E[T'] = \E[T'_A]$, $A\subsetneq [n]$.

\begin{remark}{\rm An immediate application of \eqref{e:d} implies that ${\bf Var}(f(X)) = \E[T]$. We stress that the random variables $T_A$ and $T$ already appear in \cite{Cha08}, in the context of normal approximations in the Wasserstein distance. Our use of the class of random objects $\{T', T'_A : A\subsetneq [n]\}$ for deducing bounds in the Kolmogorov distance is new.
}
\end{remark} 

The next statement is the main abstract finding of the paper.

\begin{theorem}
\label{thm:SchulteSchatterjeebound}
Let the assumptions and notation of the present section prevail, let $N\sim \mathcal{N}(0,1)$, and assume that $\E W = 0$ and $\E W^2 =\sigma^2\in (0,\infty)$. Then,
\begin{align}
\label{eq:abstract-intermed-bound}  d_{K}(\sigma^{-1} W,N)&\leqslant \frac{1}{\sigma^2} \sqrt{\var(\E(T | X))}+\frac{1}{\sigma^2}\sqrt{\var(\E\left( T'| X \right))}\\
&\quad\quad\quad \quad+\frac{1}{4\sigma ^{4}}\E\sum_{j,A,j\notin A} \kappa _{n,A}   | f(X) |   \left| \Delta _{j}f(X,X')  ^{2}    \Delta _{j}f(X^A,X')  \right| \notag\\ &\quad\quad\quad\quad\quad\quad\quad \quad\quad\quad\quad\quad \quad\quad\quad\quad\quad \quad\quad+ \frac{\sqrt{2\pi }}{16 \sigma^{3}}\sum_{j=1}^{n}\E  | \Delta _{j}f(X,X') | ^{3} \notag\\
\label{eq:abstract-bound}&\leqslant \frac{1}{\sigma^2} \sqrt{\var(\E(T | X))} +\frac{1}{\sigma^2}\sqrt{\var(\E\left( T'| X \right))}\\
 &\quad\quad\quad\quad\quad\quad+\frac{1}{4\sigma^3} \sum_{j=1}^{n}\sqrt{\E  | \Delta _{j}f(X,X') | ^{6}}+\frac{\sqrt{2\pi }}{16 \sigma^{3}}\sum_{j=1}^{n}\E  | \Delta _{j}f(X,X') | ^{3} .\notag
\end{align}
\end{theorem}

\begin{proof}
By homogeneity, we can assume that $\sigma=1$, without loss of generality. By virtue of \eqref{e:upkol}, the Kolmogorov distance between $W$ and $N$ is the supremum over $t\in [0,1]$ of
\begin{align}
\label{eq:11}
 |\E  g_{t}'(W)-Wg_{t}(W)|& \leqslant  \E | g_{t}'(W)-g_{t}'(W)T |+ |\E (g_{t}(W)W- g_{t}'(W)T) |,
\end{align}
where the derivative $g'_t(w)$ is defined for every real $w$, thanks to the convention \eqref{e:convention}.
Since $W$ is $\sigma(X)$-measurable, $ | g_{t}' | \leqslant 1$ and $\E T=\E W^{2}=1$, one infers that
\begin{align*}
\E | g_{t}'(W)-g_{t}'(W)T |&\leqslant \E  [| g_{t}'(W) \times  \E[ T-1 \;| \; X] | ]\leqslant \E  | \E  [ T-1   \;| \;  X] |  \leqslant\sqrt{ \var(\E(T | X))}.
\end{align*} 
Our aim is now to show that the quantity $|\E (g_{t}(W)W- g_{t}'(W)T) |$ is bounded by the last three summands on the right-hand side of \eqref{eq:abstract-intermed-bound} (with $\sigma=1$). Reasoning as in \cite{Cha08}, the relation (\ref{e:d}) applied to $\E g_{t}(W)W$ and the definition of $T$ yield

\begin{align*}
 |  \E g_t (W)W-g_t '(W)T|&=  \left|  \frac{1}{2}\sum_{A\subsetneq [n] }\kappa _{n,A}\sum_{j\notin A}\E (R_{A,j}-\tilde R_{A,j})\right|  \\
 &\leqslant  \frac{1}{2}\sum_{A\subsetneq [n] }\kappa _{n,A}\sum_{j\notin A}\E | R_{A,j}-\tilde R_{A,j}|,
\end{align*}
 with
\begin{align*}
R_{A,j}&=\Delta _{j}((g_t \circ f)(X))\Delta _{j}f(X^A),\\
\tilde R_{A,j}&=g_t '(f(X))\Delta _{j}f(X)\Delta _{j}f(X^A),
\end{align*}
where, here and for the rest of the proof, we use the simplified notation $\Delta _{j}f(X^A) = \Delta _{j}f(X^A,X')$, $\Delta _{j}f(X) = \Delta _{j}f(X,X')$, and so on. We have
\begin{align*}
  \E | R_{A,j}-\tilde R_{A,j} |  & = \E \big[ | g_{t}(f(X)- \Delta _{j}f(X))-g_{t}(f(X))- g_{t}'(f(X))( -\Delta _{j}f(X)  ) | \times | \Delta _{j}f(X^A ) | \big].
\end{align*}
Now we use (\ref{eq:stein2}) with $w=f(X), h= - \Delta _{j}f(X)$, together with the fact that  
\begin{align*}
h\left(\one_{ \left[ w , w + h \right)}(t) - \one_{\left[  w+h , w\right)}(t)\right)
 = - h (\one_{\{w>t\}} - \one_{\{w+h>t\} })
\end{align*}
to deduce that
\begin{align}
\label{eq:bound1}
 |  \E [g_t (W)W-g_t '(W)T] |\leqslant \frac{1}{2} \E \sum_{j,A,j\notin A} \kappa _{n,A}&\Big\{ \big(  | f(X)  | +\sqrt{2\pi }/{4} \big)\frac{ | \Delta _{j}f(X) | ^{2} | \Delta _{j}f(X^A) |}{2}  \\
\nonumber &  + \Delta _{j}\left(\one_{f(X)>t}\right)\Delta _{j}f(X) \left| \Delta _{j}f(X^A) \right|  \Big\}.
\end{align}
Using the independence of $X$ and $X'$, one proves immediately that, for $j\notin A$,
\begin{align*}
\E \Delta _{j}\left(\one_{f(X)>t}\right)\Delta _{j}f(X) \left| \Delta _{j}f(X^A) \right| = 2\E \one_{f(X)>t} \Delta _{j}f(X) \left| \Delta _{j}f(X^A) \right|,
\end{align*}
from which it follows that the right-hand side of (\ref{eq:bound1}) is bounded by 
\begin{align*}
&\frac{1}{4}\E\sum_{j,A,j\notin A} \kappa _{n,A}  \left(  | f(X) | +\frac{\sqrt{2\pi }}{4} \right)  \left| \Delta _{j}f(X)  ^{2}    \Delta _{j}f(X^A)  \right|  +\left | \E\left[  \one_{f(X)>t}\times T'  \right] \right  | \\
&\leqslant \frac{1}{4}\E\sum_{j,A,j\notin A} \kappa _{n,A}  \left(  | f(X) | +\frac{\sqrt{2\pi }}{4} \right)  \left| \Delta _{j}f(X)  ^{2}    \Delta _{j}f(X^A)  \right| +\sqrt{{\bf Var} ({\bf E} (T'\, |\, X))},
\end{align*}
where we have applied the Cauchy-Schwartz inequality, together with the fact that indicator functions are bounded by 1. The  bound (\ref{eq:abstract-intermed-bound}) is obtained by using the H\"older inequality in order to deduce that, for all $j,A$,
\begin{align*}
 \E| \Delta _{j}f(X) |^{2}  | \Delta _{j}f(X^A) |&\leqslant \E  | \Delta _{j}f(X) | ^{3},
\end{align*}and (\ref{eq:abstract-bound}) follows by 
\begin{align*}
  \E  | f(X) |  | \Delta _{j}f(X) | ^{2} | \Delta _{j}f(X^A) | &\leqslant \sqrt{\E f(X)^{2}}\sqrt{\E  \Delta _{j}f(X)^{4}\Delta _{j}f(X^A)^{2}}\\
&\leqslant  \sqrt{( \E \Delta _{j}f(X)^{4(3/2)})^{2/3} (\E \Delta _{j}f(X^A))^{2(3)})^{1/3} }\leqslant ( \E \Delta _{j}f(X)^{6})^{1/2},\\
\end{align*}
where we have used the fact that $X$ and $X^A$ have the same distribution.
\end{proof}

\begin{remark}{\rm

Recall that the {\it Wasserstein distance} between the laws of two real-valued random variables $U,V$ is defined as
$$
d_{W}(U,V) := \sup_h \left| \E[h(U)] - \E[h(V)] \right|,  
$$
where the supremum runs over all 1-Lipschitz functions $h : \mathbb{R} \to \mathbb{R}$. In \cite[Theorem 2.2]{Cha08}, one can find the following bound: under the assumptions of Theorem \ref{thm:SchulteSchatterjeebound},
\begin{equation}
\label{e:souravbound}  d_{W}(W,N)\leqslant \frac{1}{\sigma^2} \sqrt{\var(\E(T | X))}+\frac{1}{2\sigma^3}\sum_{j=1}^{n}\E  | \Delta _{j}f(X,X') | ^{3}.
\end{equation}
}
\end{remark} 
 
\begin{example}{\rm Consider a vector $X =(X_1,...,X_n)$ of i.i.d. random variables with mean zero and variance 1, and assume that $\E | X_1|^4<\infty$. Define $W = f(X) = n^{-1/2}(X_1+\cdots +X_n)$. It is easily seen that, in this case, for every $j\notin A$, $\Delta_j f(X^{A}, X') = n^{-1/2} (X_j - X'_j)$, in such a way that
$$
T = \frac{1}{2n} \sum_{j=1}^n (X_j-X'_j)^2 \quad \mbox{and} \quad T' = \frac{1}{2n} \sum_{j=1}^n {\rm sign}(X_j-X'_j)(X_j-X'_j)^2.
$$
We also have, denoting $\hat X^j$ the vector $X$ after removing $X_{j}$,
\begin{align*}
\E  | f(X)\Delta _{j}f(X)^{2}\Delta _{j}f(X^{A}) |& \leqslant \E  | f(X)-f(\hat X^j) | |  \Delta _{j}f(X)^{2}\Delta _{j}f(X^{A}) | +\E  | f( \hat X^j) | \E  | \Delta _{j}f(X)^{2} |  | \Delta_{j}f(X^{A}) | \\
&\leqslant \E n^{-2} | X_{j} | | X_{j}-X_{j}' |^{2} | X_{j}-X_{j}' |   +\E | f(\hat X^{j}) |\E n^{-3/2} | X_{j}-X_{j}' |^{2} | X_{j}-X_{j}' |   \\
&\leqslant  8(n^{-2}\E X_{j}^{4}+n^{-3/2}\E X_{j}^{3}).
\end{align*}
({\blu note that} the bound \eqref{eq:abstract-bound} can be used instead, whenever $\E X_{1}^{6}<\infty $).
An elementary application of \eqref{eq:abstract-intermed-bound} yields therefore that there exists a finite constant $C>0$, independent of $n$, such that
$$
d_K(W,N)\leqslant \frac{C}{\sqrt{n}},
$$
providing a rate of convergence that is consistent with the usual Berry-Esseen estimates. One should notice that the estimate \eqref{e:souravbound} yields the similar bound $d_{W}(W,N)\leqslant C/\sqrt{n}$.

}

\end{example}

} 
 
\section{Symmetric functions and geometric applications}

In this section we adapt our results to random structures with local dependence, {in a spirit close} to \cite[Section 2.3]{Cha08} -- see Remark \ref{r:zezz} below. {Our principal focus will be on measurable and symmetric real-valued mappings $f$ on $E^{n}$: we recall that $f : E^n \to \mathbb{R}$ is said to be {\it symmetric} if 
\begin{align*}
f(x_{\sigma (1)},\dots ,x_{\sigma (n)})=f(x_{1},\dots ,x_{n})
\end{align*} 
for any permutation $\sigma $ of $[n]$ and vector $x\in E^{n}$.


\smallskip

In the following, $X$ and $X'$ denote two independent sets of $n$ i.i.d. random variables with {\blu common} generic distribution $\mu $. We will use the following short-hand notation: for any random vector $Z$ {of dimension $n$}, and for every $1\leqslant i\neq j\leqslant n$,
\begin{align*}
\Delta _{i}f(Z):=\Delta _{i}f(Z,X'),\; \Delta_{i,j}f(Z){\orange :=\Delta _{i}\Delta_j f(Z,X')},
\end{align*}
where the notation is the same as in Section \ref{ss:difference}; we also adopt the additional convention that $\Delta _{i,i}=\Delta _{i}$. Now let $\tilde X$ be a further independent copy of $X$. We shall use the following terminology: a vector $Z=(Z_1,...,Z_n)$ is a {\it recombination} of $\{X,X',\tilde X\}$, if $Z_{i}\in \{X_{i},X'_{i},\tilde X_{i}\}$ for every $1\leqslant i\leqslant n$.

\smallskip

The next statement provides a bound for the normal approximation of geometric functionals that is amenable to geometric analysis, {\blu and can be heuristically regarded as the binomial counterpart to the second order Poincar\'e inequalities on the Poisson space (in the Kolmogorov distance), proved in \cite{LPS}}.

\begin{theorem}
  \label{thm:dependancy}
  Let $f:E^{n}\to \mathbb{R}$ be a symmetric measurable functional such  that $W=f(X)$ is centred, and $ \sigma ^{2}=\var(W)<\infty $.  Let $N$ be a centred Gaussian random variable with variance 1. Define
\begin{align*}
 \coef&:=\sup_{(Y,Z,Z')}\E\left[ \mathbf{1}_{\{\Delta _{1,2}f(Y)\neq 0\}}\Delta _{1}f(Z)^{2}\Delta _{2}f(Z')^{2} \right],  \\
   \coeff&:=\sup_{(Y,Y',Z,Z')}\E\left[ \mathbf{1}_{\{\Delta _{1,2}f(Y)\neq 0,\Delta _{1,3}f(Y')\neq 0\}}\Delta _{2}f(Z)^{2}\Delta _{3}f(Z')^{2} \right],  \\
\end{align*}where the suprema run over all vectors  $Y,Y',Z,Z'$ that are recombinations of $\{X,X',\tilde X\}$. Then,
\begin{align}
\label{eq:graph-bound} 
d_{K}( \sigma^{-1}W, N) & \leqslant \left[\frac{{\rosso 4}\sqrt{2}  n^{1/2}}{\sigma^2} \left( \sqrt{n\coef  } +\sqrt{n^{2}\coeff}+ \sqrt{\E \Delta _{1}f (X)^{4}}\right) \right.\\
\notag &\left. \hspace{2cm}+{\blu \frac{n}{4\sigma^4}}\sup_{A\subseteq [n] } \E  | f(X)  {\blu    \Delta  _{1}f(X^{A}) ^{3} }  | +\left(  \frac{ \sqrt{{\blu2} \pi } }{16\sigma ^{3}}n\E  |   \Delta  _{1}f(X )^{3} |  \right)   \right] .
\end{align}  
\end{theorem}
 \begin{remark}
We shall often use the following bounds, following at once from the Cauchy-Schwartz inequality,
\begin{align}
\label{eq:coeff-CS}
\notag\coeff& \leqslant \sup_{(Y,Y',Z,Z')}\sqrt{ \E\left[ \mathbf{1}_{\{\Delta _{1,2}f(Y)\neq 0,\Delta _{1,3}f(Y')\neq 0\}}\Delta _{2}f(Z)^{4} \right]\E\left[ \mathbf{1}_{\{\Delta _{1,2}f(Y)\neq 0,\Delta _{1,3}f(Y')\neq 0\}}\Delta _{3}f(Z')^{4} \right]}\\
&\leqslant\sup_{(Y,Y',Z)} \E\mathbf{1}_{\{\Delta _{1,2}f(Y)\neq 0,\Delta _{1,{\rosso 3}}f(Y')\neq 0\}}\Delta _{\blu 2}f(Z)^{\blu 4}
\end{align}and
\begin{align}
\label{eq:coef-CS}
\coef\leqslant\sup_{(Y,Z)} \E \left[ \mathbf{1}_{\{\Delta _{1,2}f(Y)\neq 0\}}\Delta _{1}f(Z)^{4} \right].
\end{align}{\blu In the framework of} the applications developed in this paper, such estimates {\blu simplify} some computations and do not {\blu worsen the associated rates of convergence}.
\end{remark}

  In the applications developed below, we will often consider functions $f$ that are obtained as restrictions to $E^n$ of general real-valued mappings on the set $\cup _{n\geq 1}E^{n}$, corresponding to the class of all finite ordered point configurations (with possible repetitions). Now fix $f : \cup _{n\geq 1}E^{n}\to \mathbb{R}$ and, for every $n\geq 1$ and every $x= (x_1,...,x_n) \in E^n$, introduce the notation $\hat{x}^{ i}$ to indicate the element of $E^{n-1}$ obtained by deleting the $i$th coordinate of $x$, that is: $\hat{x}^{ i} = (x_1,...,x_{i-1},x_i,...,x_n)$. {\blu Analogously, write} $\hat{x}^{ij}\in E^{n-2}$ to denote the vector obtained from $x$ by removing its $i$-th and $j$-th coordinates.
We write
\begin{align*}
D_{i}f(x)&=f(x)-f(\hat{x}^{ i}),\\
D_{i,j}f(X)&=f(x)-f(\hat x^{i})-f(\hat x^{j})+f(\hat x^{ij})=D_{j,i}f(x).
\end{align*}
 
 \begin{proposition}
 \label{r:zee}
 Let $f$ be a functional defined on $\cup _{k\leqslant n}E^{k}$ such that its restriction to $E^{n}$ satisfies the hypotheses of Theorem \ref{thm:dependancy}. Then we have 
\begin{align*}
B'_{n}(f)&\leqslant { 2^{8}}\sup_{(Y,Y',Z,Z')}\E \left[\mathbf{1}_{\{ D_{1,2}f(Y)\neq 0\}}\mathbf{1}_{\{D_{1,3}f(Y')\neq 0\}}D_{2}f(Z)^{2}D_{2}f(Z')^{2} \right]\\
B_{n}(f)&\leqslant { 2^{6}}\sup_{(Y,Z,Z')}\E \left[\mathbf{1}_{\{ D_{1,2}f(Y)\neq 0\}} D_{1}f(Z)^{2}D_{2}f(Z')^{2} \right].\\
\end{align*}
 \end{proposition}

\begin{proof}First observe that
\begin{align}
 | \Delta _{j}f(X) | &\leqslant  | D_{j}f(X) | + | D_{j}f(X^{j}) | \label{e:zu} \\
   \Delta _{i,j}f(X) &= D_{i,j}f(X)-D_{i,j}f(X^{i})    -D_{i,j}f(X^{j})+D_{i,j}f(X^{\{i,j\}}). \label{e:zuz}
\end{align}  
Let $Y,Y',Z,Z'$ be recombinations of $\{X,X',\tilde X\}$. Using the bounds above, there are recombinations $Y^{(i)},Y^{',(i)},i=1,\dots ,4$ and $Z^{(l)},Z^{',(l)},l=1,2$, such that
\begin{align*}
\E&\left[ \mathbf{1}_{\{\Delta _{1,2}f(Y)\neq 0,\Delta _{1,3}f(Y')\neq 0\}}\Delta _{2}f(Z)^{2}\Delta _{3}f(Z')^{2} \right]\\
&\leqslant \E\left[ \sum_{i=1}^{4}\mathbf{1}_{\{D_{1,2}f(Y^{(i)})\neq 0\}}\sum_{j=1}^{4}\mathbf{1}_{\{D_{1,3}f(Y^{',(j)})\neq 0\}}\sum_{l,m=1}^{2} { 4}D_{2}f(Z^{(l)})^{2}D_{3}(Z^{',(m)})^{2}\right]\\
&\leqslant { 256}\sup_{(Y,Y',Z,Z')}\E \left[ \mathbf{1}_{\{D_{1,2}f(Y)\neq 0\}}\mathbf{1}_{\{D_{1,3}f(Y')\neq 0\}}D_{2}f(Z)^{2}D_{3}f(Z')^{2} \right],
\end{align*}which gives the bound on $B_{n}'(f)$. The bound on $B_{n}(f)$ is obtained analogously.
\end{proof}

\begin{remark}\label{r:zezz}
Our framework is more restrictive than that of {\orange \cite[Theorem 2.5]{Cha08}}, where it is not assumed that $f$ is symmetric, but rather that its dependency graph {is symmetric}, meaning that the relation $\Delta _{i,j}f(X)=0$ is equivalent to $\Delta _{\sigma (i),\sigma (j)}f(X^{\sigma })=0$ for any $i\neq j$ and every permutation $\sigma $ of $\{1,...,n\}$, where $X^{\sigma }_{i}:=X_{\sigma (i)}$. One should notice that { this subtlety is not exploited} in most applications of \cite{Cha08} -- see e.g. \cite{Nol}.
Under our symmetry assumption, {\blu a bound analogous to the main estimate in} {\orange \cite[Theorem 2.5]{Cha08}} can be retrieved from \eqref{eq:graph-bound} by using the bounds
\begin{align*}
\sqrt{\E \Delta _{j}f(X)^{4}} \,+\, &\sqrt{n\coef }+\sqrt{n^{2}\coeff}\\ &\leqslant 3\sqrt{\E\Delta _{j}f(X)^{4}+n\coef  +n^{2}\coeff}\\
&\leqslant 3\sqrt{8\sum_{j,k=1}^{n}\sup_{(Y,Y',Z,Z')}\E1_{\{\Delta _{1,j}f(Y)\neq 0\}}1_{\{\Delta _{1,k}f(Y')\neq 0\}}\Delta _{j}f(Z)^{2}\Delta _{k}f(Z')^{2}}\\
&\leqslant 6\sqrt{2}\sqrt{\sum_{j,k=1}^{n}\sup_{(Y,Y',Z)}n^{-2}\E (\sup_{j=1}^{n} | \Delta _{j}f(Z) | )^{4}\delta _{1}(Y)\delta _{1}(Y')}\\
&\leqslant 6\sqrt{2} (\E M(X)^{8})^{1/4}(\E \delta _{1}(X)^{4})^{1/4}
\end{align*}where $M(X)=\sup_{i} | \Delta _{i}f(X) | $ and $\delta _{1}(X)=\#\{j:\Delta _{1,j}f(X)\neq 0\}$. One should notice that the additional term {\blu involving quantities of the type} $\E  | f(X)  \Delta  _{1}f(X) ^2  \Delta  _{1}f(X^{A})    |$ {\blu appears in our bounds because we are dealing with the} Kolmogorov distance: in general, we shall control this term by using the rough estimate $\E  | f(X)  \Delta  _{1}f(X) ^2  \Delta  _{1}f(X^{A})  |\leqslant \sigma \sqrt{\E \Delta _{j}f(X)^{6}} $, {\blu that one can e.g. deduce by applying twice the Cauchy-Schwartz inequality} -- see Section \ref{sec:applis} for more details.
\end{remark}

}

\begin{proof}[Proof of Theorem \ref{thm:dependancy}]
Assume without loss of generality that $\sigma =1$. {\blu Our estimate follows by appropriately bounding each of the four summands appearing on the right-hand side of (\ref{eq:abstract-intermed-bound}). We have for $A\subseteq [n],1\leqslant j\leqslant n,$ by H\"older inequality,
\begin{align*}
\E |  f(X)\Delta _{j}f(X)^{2}\Delta _{j}f(X^{A}) |&= \E  | f(X)^{2/3} \Delta _{j}f(X)^{2}| | \Delta _{j}f(X)^{1/3}\Delta _{j}f(X^{A}) |\\
&\leqslant \left( \E  | f(X) \Delta _{j}f(X)^{3}|  \right)^{2/3}\left( \E  | f(X)\Delta _{j}f(X^{A}) ^{3}|  \right)^{1/3}\\
&\leqslant \sup_{A\subseteq [n]}\E  | f(X)\Delta _{j}f(X^{A})^{3} |,    
\end{align*}because $\Delta _{j}f(X)=\Delta _{j}f(X^{\emptyset })$. The  two last terms on the right-hand side of (\ref{eq:abstract-intermed-bound}) are therefore bounded by the  last two terms in \eqref{eq:graph-bound}, in view of the symmetry of $f$} and of the relation $\sum_{A\subsetneq [n]:1\notin A}{\blu \kappa_{n,A}}=1$. 
 To control the first two summands in (\ref{eq:abstract-intermed-bound}), we first bound the square root of the variance of a random variable of the type $U:={ \frac12}\sum_{A\subsetneq [n]}\kappa _{n,A}U_{A}$, for a general family of square-integrable random variables $U_{A}(X,X'), A\subsetneq [n]$.  Using e.g. \cite[Lemma 4.4]{Cha08}, we infer that
\begin{align}
\label{eq:init-preuve-graphe}
\sqrt{\var\left( \E\left( U | X \right) \right)}&\leqslant \frac{1}{2} \sum_{A\subsetneq   [n]}\kappa _{n,A}\sqrt{\var\E\left( U_{A} | X \right)}\leqslant \frac{1}{2}\sum_{A\subsetneq [n]}\kappa _{n,A}\sqrt{\E\left( \var(U_{A} | X') \right)}.
\end{align} 
This inequality will be used both for $U_{A}=T_{A}$ and $U_{A}=T_{A}'$.
Let us now bound each summand separately. Fix $A\subseteq [n]$. Introduce  the  substitution operator based on $\tilde X=(\tilde X_{i})_{1\leqslant i\leqslant n}$  
\begin{align*}
\tilde S_{i}(X)=(X_{1},\dots ,\tilde X_{i},\dots ,X_{n} ).
\end{align*}  Recall that, by the Efron-Stein's inequality, for any  square-integrable functional $Z(X_{1},\dots ,X_{n})$,
\begin{align*}
\var(Z)\leqslant \frac{1}{2}\sum_{i=1}^{n}\E(\tilde \Delta _{i}Z(X))^{2}
\end{align*}where 
\begin{align*}
(\tilde \Delta _{i}Z)(X):=Z(\tilde S_{i}(X))-Z(X)
\end{align*}is clearly centred.
Applying this to $Z(X)=U_{A}(X,X')$ for fixed $X'$, 
\begin{align*}
\var(U_{A} | X')  &\leqslant \frac{1}{2}\sum_{i=1}^{n}\E\left[\left( \tilde \Delta _{i}U_{A}(X,X') \right) ^{2} | X'\right].
\end{align*}
 From this relation, we therefore infer that 
\begin{align*}
\sqrt{\var\left( \E\left( U | X \right) \right)}\leqslant {\orange \frac{1}{\blu \sqrt{8}}\sum_{A\subsetneq [n]}\kappa _{n,A} \sqrt{\sum_{i=1}^{n}\E \left( \tilde \Delta _{i}U_{A} \right)^{2}.}}
\end{align*}
 Now recall that  $U_A=T_{A}$ or $U_{A}=T_{A}'$, i.e. $U_{A} =\sum_{j\notin A}\Delta _{j}f(X)g(\Delta _{j}f(X^{A}))$, where either $g$ is the identity or $g(\cdot )= | \cdot  | $. 
Expanding the square yields
\begin{align}
\label{eq:square-var-UA}
\sum_{i=1}^{n}{\blu \E \left( \tilde \Delta _{i}U_{A} \right)^{2}}&= \sum_{i=1}^{n}\sum_{j,k\notin A} \E | \tilde \Delta _{i}(\Delta _{j}f(X)g(\Delta _{j}f(X^{A}))) || \tilde \Delta _{i}(\Delta _{k}f(X)g(\Delta _{k}f(X^{A}))) |.\end{align}
{Now fix $1\leqslant i\leqslant n$, write $\tilde X^{i}=\tilde S_{i}(X)$ and observe that for $j\notin A$,
\begin{align}
\label{eq:derivative-product}
  \tilde \Delta _{i}(\Delta _{j}f(X)g(\Delta _{j}f(X^{A}))) =\tilde \Delta _{i}(\Delta _{j}f(X))g(\Delta _{j}f(X^{A}))+\Delta _{j}f(\tilde X^{i})\tilde \Delta _{i}(g(\Delta _{j}f(X^{A}))).\end{align}
We note immediately that, in the case $i=j$, using $ | \tilde \Delta _{i}g(V(X)) | \leqslant  | \tilde \Delta _{i}(V(X)) | $ and $\tilde \Delta _{i}(\Delta _{i}(V(X)))=\tilde \Delta _{i}(V(X))$ for any random variable $V(X)$, {\blu the right-hand side of \eqref{eq:derivative-product}} is bounded by the simpler expression
\begin{align}
\label{eq:product-derivative-i=j}
 | \tilde \Delta _{i} f(X)\Delta _{i}f(X^{A}) | + | \Delta _{i}f(\tilde X^{i})\tilde\Delta _{i}f(X^{A}) |  \leqslant \frac{1}{2}\left[ \tilde \Delta _{i}f(X)^{2}+\Delta _{i}f(X^{A})^{2}+\Delta _{i}f(\tilde X^{i})^{2}+\tilde \Delta _{i}f(X^{A})^{2} \right].
\end{align}
Now let us examine each summand appearing in \eqref{eq:square-var-UA} separately. If $i\notin A$ and $i=j=k$, using \eqref{eq:product-derivative-i=j}, the summand is smaller than 
\begin{align*}
\frac{1}{4}\E\left[ \tilde \Delta _{i}f(X)^{2}+\Delta _{i}f(X^{A})^{2}+\Delta _{i}f(\tilde X^{i})^{2}+\tilde \Delta _{i}f(X^{A})^{2} \right]^{2}\leqslant 4\E \Delta _{1}f(X)^{4} .
\end{align*}
In the case where $i,j,k$ are pairwise distinct,  introduce the vector $\bar X$ by 
\begin{align*}\left\{\begin{array}{ll}
\bar X_{i}&=\tilde X_{i}\\
\bar X_{l}&=X_{l}' \text{ if }l\neq i,\end{array}\right.
\end{align*}and, for $x\in E^{n}$ and some {\blu mapping} $\psi $ on $E^n$, define, for $1\leqslant l\leqslant n$, $$\bar \Delta _{l}\varphi (x)=\psi  (x)-\psi  (x_{1},\dots ,x_{l-1},X_{l},x_{l+1},\dots ,x_{n}).$$
Then, the corresponding summands are bounded by\begin{align*}
4&\sup_{(Y,Y',Z,Z')}\E\left|  \bar \Delta _{i}(\bar\Delta _{j}f(Y))\bar \Delta _{j}f(Y')\bar\Delta _{i}(\bar \Delta _{k}f(Z))\bar \Delta _{k}f(Z') \right| .
\end{align*}
 Using $\bar X\equlaw X'$ and the fact that if $Y$ is a recombination, switching the roles of $\tilde X_{i}$ and  $X'_{i}$ in $Y$ still yields a recombination of $\{X,X',\tilde X\}$, the previous expression is bounded by
\begin{align*}
&=4\sup_{(Y,Y',Z,Z')}\E\left|   \Delta _{i}(\Delta _{j}f(Y))\Delta _{j}f(Y')\Delta _{i}(\Delta _{k}f(Z))\Delta _{k}f(Z') \right| \\
&\leqslant 4\sup_{(Y,Y',Z,Z')}\E\mathbf{1}_{\{\Delta _{i,j}f(Y)\neq 0\}}( | \Delta _{j}f(Y) |+ | \Delta _{j}f(Y^{i}) |  ) | \Delta _{j}f(Y') | \times \\
&\quad\quad\quad\quad\quad\quad \times \mathbf{1}_{\{\Delta _{i,k}f(Z)\neq 0\}}( | \Delta _{k}f(Z) |+ | \Delta _{k}f(Z^{i}) |  ) | \Delta _{k}f(Z') |\\
& \leqslant 16\coeff,
\end{align*}where we have used Cauchy-Schwarz inequality. The case $i\neq j=k$ is treated with the same vector $\bar X$ and operators $\bar \Delta _{l}$. Using   similar computations  and Cauchy-Schwarz inequality, we have the upper bound 
\begin{align*}
&4\sup_{(Y,Y',Z,Z')}\E \bar \Delta _{i}(\bar\Delta _{j}f(Y))\bar \Delta _{j}f(Y')\bar\Delta _{i}(\bar \Delta _{j}f(Z))\bar \Delta _{j}f(Z')\\
&\leqslant 4\sup_{(Y,Y')}\left[ \E \bar \Delta _{i}(\bar\Delta _{j}f(Y))^{2}\bar \Delta _{j}f(Y')^{2} \right]\\
&=4\sup_{(Y,Y')}\left[ \E  \Delta _{j}(\Delta _{i}f(Y))^{2} \Delta _{j}f(Y')^{2} \right]\\
&\leqslant 4\sup_{(Y,Y')}\E \mathbf{1}_{\{\Delta _{i,j}f(Y)\neq 0\}}( | \Delta _{i}f(Y) |+ | \Delta _{i}f(Y^{j}) |  )^{2}\Delta _{j}f(Y')^{2}\\
&\leqslant 16\sup_{(Y,Y',Z)}\E\mathbf{1}_{\{\Delta _{i,j}f(Y)\neq 0\}} \Delta _{i}f(Z)  ^{2}\Delta _{j}f(Y')^{2}\\
&\leqslant 16\coef,
\end{align*} where the suprema run over  recombinations $Y,Y',Z,Z'$ of $\{X,X',\tilde X\}$. Finally, if $i=j\neq k$,   the corresponding summands on the right-hand side of \eqref{eq:square-var-UA} are bounded by 
\begin{align*}
4\sup_{(Y,Y',Z)}\E&\left| \bar \Delta _{i}f(Y)^{2}\bar \Delta _{i}(\bar \Delta _{k}f(Y'))\bar \Delta _{k} f(Z) \right| \\
&\leqslant 4\sup_{(Y,Y',Z)}\E \mathbf{1}_{\{\Delta _{i,k}f(Y')\neq 0\}}( | \Delta _{k}f(Y) |+ | \Delta _{k}f(Y^{i}) |  )\Delta _{i}f(Y)^{2} | \Delta _{k}f(Z) |\\
&\leqslant 8\coef. 
\end{align*}
This yields
\begin{align*}
{\blu \sum_{i=1}^{n}\E \left( \tilde \Delta _{i}U_{A} \right)^{2}}
&\leqslant {\blu 16}n\sum_{j,k\notin A}\left[ \mathbf{1}_{\{j=k=1\}}\E\Delta _{1}f(X)^{4}+(\mathbf{1}_{\{k\neq j=1\}}+\mathbf{1}_{\{k=j\neq 1\}})B_{n}(f)+\mathbf{1}_{\{k\neq j\neq 1\}}B_{n}'(f)  \right]\\
&\leqslant {\blu 16}n\left(\mathbf{1}_{\{1\notin A\}} \Delta _{1}f(X)^{4}+2(n- | A | )B_{n}(f)+(n- | A | )^{2}B_{n}'(f)   \right),
\end{align*} 
and using the inequality  $\sqrt{x+y}\leqslant \sqrt{x}+\sqrt{y}$ ($x,y\geq 0$) we deduce that 
\begin{align*}
{\blu\sqrt{ \sum_{i=1}^{n}\E \left( \tilde \Delta _{i}U_{A} \right)^{2}}}&\leqslant \sqrt{{\blu 16}n}\left(  \mathbf{1}_{\{1\notin A\}}\sqrt{\E \Delta _{1}f(X)^{4}}+\sqrt{2{\rosso B_{n}(f)}}\sqrt{n- | A | }+\sqrt{{\rosso B_{n}'(f)} }(n- | A | )  \right).\\ 
\end{align*}
Finally,
\begin{align*}
&\sqrt{\var(\E(U |  X ))}\leqslant \\
&\sqrt{{\blu 8}n}\left( \sqrt{\E\Delta _{1}f(X)^{4}}\sum_{A\subsetneq [n]:1\notin A}\kappa _{n,A}+\sqrt{{\rosso B_{n}(f)}}\sum_{A\subsetneq [n]}\kappa _{n,A}\sqrt{n- | A | }+\sqrt{{\rosso B_{n}'(f)}}\sum_{A\subsetneq [n]}\kappa _{n,A}(n- | A | ) \right)
\end{align*}and the result follows by evaluating the three sums over $A\subsetneq [n]$ in the last expression.

}

\end{proof}

\section{Applications}
\label{sec:applis}

  \subsection{Set approximation with random tessellations}
 \label{sec:voronoi}

Let $K$ be a compact subset of $ \mathbb{R}^d$ with positive volume, and let $X=(X_{i})$ be a locally finite collection of points. Assume the only information available about $K$ is given by the values of the indicator function $1_{\{x\in K\}}, x\in X$. Then, the {\it Voronoi reconstruction}, or {\it Voronoi approximation}, of $K$ based on $X$ is  defined as
\begin{align*}
K^{X}=\{y\in \mathbb{R}^{d}:\text{ the closest point from $y$ in $X$ lies in $K$}\}.
\end{align*}  
This chapter is devoted to the study of the error committed when one  approximates the volume of $K\subseteq [0,1]^{d}$ with that of $K^{X}$, when $X$ is a random input consisting in  $n$ i.i.d points  in $[0,1]^d$.
 
 The underlying structure in this approximation scheme is the Voronoi tessellation based on   $X $. For $x\in [0,1]^d$,
denote by   $V(x;X)$ the Voronoi cell with nucleus $x$ among $X$, i.e. the convex set  formed by points $y\in [0,1]^{d}$ such that $\|y-x\|\leqslant \|y-x'\|$ for any point $x'\in (X,x)$, where in all this section { $(X,x) := X\cup \{x\}$ }, and we extend the set notation $\in $ to ordered collections of points in an obvious way. The volume approximation described above is denoted 
\begin{align*}
\varphi(X)=\Vol(K^{X})=\sum_{i}1_{\{X_{i}\in K\}}\Vol(V(X_{i};X)).
\end{align*}Along the same lines, one can also approximate the perimeter of $K$ via the relation 
 $
\varphi _{\Per}(X)=\Vol(K^{X}\Delta K) $
 where $\Delta $ denotes the symmetric difference of sets.  

This set approximation can serve in image reconstruction and estimation: it has first been introduced by Einmahl and Khmaladze \cite{EinKhm} as a discriminating statistic in the two-sample problem. These authors proved a strong law of large numbers in dimension $1$.   Heveling and Reitzner  \cite{HevRei} proved that if $K$ is convex and compact and $X=X'$ is a homogeneous Poisson process with  intensity $n$, $\E \varphi(X')=\Vol(K)$, and $\var(\varphi(X'))\leqslant cn ^{-1-1/d}S(K)$ where $c$ is an explicit constant and $S(K)$ is the surface area of $K$. They  also established  that $\E \varphi _{\Per}(X')=c'n ^{-1/d}S(K)(1+O(n ^{-1/d}))$ and $\var(\varphi _{\Per}(X'))\leqslant c'n ^{-1-1/d}S(K)$. Reitzner, Spodarev and Zaporozhets \cite{ReiSpoZap}  extended these results to sets  with finite variational perimeter, and also gave upper bounds for $\E  | \varphi(X')^{q}-\Vol(K)^{q} | $ for $q\geq 1$. Schulte \cite{Sch14} proved a similar lower bound for the variance, i.e. $CS(K)n ^{-1-1/d}\leqslant \var(\varphi(X'))$ with $K$ a convex body and $C$ a universal constant, and the corresponding CLT
\begin{align*}
d_{W}\left( \frac{\varphi(X')-\E \varphi(X')}{\sqrt{\var(\varphi  (X')})} ,N\right)\to 0.
\end{align*}
Yukich \cite{Yuk15} then gave an upper bound on the speed of convergence in Kolmogorov distance.

For Binomial input, Penrose proved that for measurable $K$ and   $X$ consisting in $n$ iid variables with density $\kappa (x)>0$ on $[0,1]^{d} $, 
\begin{align}
\label{eq:penrose}
\E \varphi(X)\to \Vol(K),
\end{align}without assumption on $K$, not even the negligibility of its boundary. {Yukich \cite{Yuk15} managed to extend to a non-Poissonized setting the estimates on the variance magnitude as well as the central limit theorem for the Volume approximation. } See also \cite{CalChe14} for a result involving the Hausdorff distance.

In this section, we consider a binomial input $X=(X_{1},\dots ,X_{n})$, where the $X_{i}$ are $n$ iid variables uniformly distributed on $[0,1]^{d}$.  We give asymptotic upper bounds   for the moments of $\varphi (X) -\E\varphi(X)$,   as well as a central limit theorem with rates of convergence in the Kolmogorov distance, that is new in the literature. Note that, in the words of Heveling and Reitzner \cite{HevRei},``the general problem whether $ K^{X}$ approximates $K$ for complicated sets seems to be difficult'', and many applications of set approximation are concerned with the detection or approximation of sets with an irregular boundary, see for instance \cite{CueFraRod07} or the survey \cite[Chap. 11]{KenMol}. Our results also hold for large classes of irregular sets, with a possibly fractal boundary. The regularity of the boundary of $K$ will be assessed in terms of the following quantities. Call below {\it Lebesgue-boundary} of $K$, written $\partial K$, the class of points $x$ such that for all $\varepsilon >0$, $\Vol(B(x,\varepsilon )\cap K)>0$ and $\Vol(B(x,\varepsilon )\cap K^{c})>0$. Let $\beta >0$. Denote by $d(x,A)$ the Euclidean distance from a point $x\in \mathbb{R}^{d}$ to a subset $
  A\subseteq \mathbb{R}^{d}$. 
  Define\begin{align*}  
  \partial K^{r}&=\{x:d(x,\partial K)\leqslant r\}\\
\partial K^{r}_{+}&=K^{c}\cap \partial K^{r}\\
\gamma (K,r )&=   \int_{\partial K^{r}_{+}}\left(\frac{ \Vol(B(x,\beta r)\cap K)}{r^{d}}\right)^{2}dx.
\end{align*}  
  $K$ is said to satisfy the \emph{weak rolling ball condition} if 
\begin{align}
\label{eq:rolling-ball}
\gamma (K):=\liminf_{r>0}\Vol(\partial K^{r})^{-1}(\gamma (K,r)+\gamma (K^{c} ,r))>0.
\end{align}

This assumption somehow implies   that either  $K$ or $K^{c}$ occupies a constant positive proportion of space  as one zooms in on  a typical point close to $\partial K$, at least in a non-negligible region of $[0,1]^{d}$. It is related to {\blu a weak form of the} \emph{rolling ball condition} used in set estimation (see for instance condition (a) of Theorem 1 in \cite{CueFraRod07}, the definition of standard sets in \cite{Rod07}, Remark 4 in \cite{Sch14}, or the survey \cite[Chap. 11]{KenMol} and references therein), where for each $x\in \partial K$ a ball of radius $\beta r$ touching $x$ should lie in $\partial (K^{c})^{r}_{+}$ or $\partial K^{r}_{+}$. In our weaker form of the condition, the ball is somehow allowed to be deformed to fit in the parallel body. It certainly allows sets which boundary is smooth in a certain sense, and does not discard a priori fractal sets.  It is proved in \cite{LacVeg} that a class of fractal sets including for instance  the $2$-dimensional Von Koch flake and antiflake satisfy the condition, as well as the hypotheses of the following theorem with $\alpha =2-s$, $s=\log(4)/\log(3)$ being the fractal dimension of the boundary.

  \begin{theorem}
\label{thm:voronoi-tcl} Let $K\subset [0,1]^d$ such that
\begin{align}
\label{eq:volume-boundary-upper-bound}
\Vol(\partial K^{r})\leqslant S_+(K) r^{\alpha }
\end{align} for some $S_+(K) ,\alpha >0$. Then for $n,q\geq 1$, 
\begin{align}
\label{eq:volume-moments-upper-bound}
\E  | \varphi(X)-\E\varphi(X) | ^{q}&\leqslant  S_+(K) C_{d,q,\alpha }n^{ -q/2-\alpha /d}, 
\end{align}for some $C_{d,q,\alpha  }>0$ explicit in the proof.
 If furthermore $K$  satisfies the weak  rolling ball condition  (\ref{eq:rolling-ball}) and 
\begin{align}
\label{eq:hyp-lower-bound-boundary}
\Vol(\partial  K^{r})\geq S_-(K) r^{\alpha }
\end{align} for some $S_-(K) >0$, then for $n$ sufficiently large
\begin{align*}
C^{-}_{d} S_-(K) \gamma  (K) \leqslant\frac{ \var(\varphi(K,X))}{ n^{-1-\alpha /d } }\leqslant C^{+}_{d}S_+(K) C_{d,2,\alpha } ,\\
\end{align*}for some $C^{-}_{d},C^{+}_{d} >0$,
and
 for every $\varepsilon >0$, there is $c_{\varepsilon }>0$ not depending on $n$ such that{
\begin{align*}
d_{K}\left( \frac{\varphi(X)-\E\varphi(X)}{\sqrt{\var(\varphi (X))}},N \right)&\leqslant c_{\varepsilon } n^{-1/2 +\alpha /2d}\log(n)^{3+\alpha /d+\varepsilon },
\end{align*}}
for $n\geq 1$.

\end{theorem}

   \begin{remarks} {\rm \begin{enumerate}
   \item {\blu The previous theorem also applies to smooth sets. Blashke's theorem (see for instance \cite[Theorem 1]{Wal99}), yields that any $\mathcal{C}^{1}$ manifold $K$ with Lipschitz normal admits inside and outside rolling balls in the traditional sense, and satisfies in particular our weak rolling ball condition. Furthermore, such a set and its complement have positive reach, which proves by Steiner formula that the upper and lower bounds \eqref{eq:volume-boundary-upper-bound}, \eqref{eq:hyp-lower-bound-boundary} are satisfied, see the pioneering  work of Federer \cite{Fed59}.  The result might still hold if the boundary is only piecewise regular, see for instance Remark 4 in \cite{Sch14}.}
    
\item If   (\ref{eq:rolling-ball})  is not satisfied, we can still get a lower bound on the variance (and therefore a rate of convergence), but its magnitude will not match that of the upper bound, see Lemma \ref{lm:variance-lower-bound}. It might be difficult for such a set to get a clear estimate of the variance. See also the counterexample in \cite{LacVeg}.
\item The constant $\beta $ in the rolling ball condition is left at our choice. The larger $\beta $, the easier it is for $K$ to verify the  condition.
\item Conditions  (\ref{eq:volume-boundary-upper-bound}) and (\ref{eq:hyp-lower-bound-boundary}) imply that $K$ has Minkowski dimension equal to $d-\alpha $, and furthermore that $K$ has lower and upper Minkowski content (see for instance \cite{LacVeg}). Self similar sets satisfy these hypotheses, and are treated in   \cite{LacVeg}, as well as some examples, such as the Von Koch flake, that also satisfies the weak rolling ball condition.  We provide as well the example of a set $K$ with lower and upper Minkowski content for $\alpha =1/2$  that does not satisfy the rolling ball condition. Simulations indicate that for this example the variance is indeed negligible with respect to  $n^{-1-\alpha /d}$, but it is still possible to get a rate of convergence for Kolmogorov distance to the normal law. 
\item The uniformity of the distribution of the $X_{i}$'s does not have a crucial importance, apart from easing certain geometric estimates. The results should hold, up to constants, if the common distribution of the $X_{i}$'s is only assumed to have a density bounded from below by some constant $\kappa >0$ on the domain $\partial K ^{r}$, for some $r >0$.
\item The Berry-Essen bounds is derived from \eqref{eq:graph-bound}. It turns out that each of the terms on the right hand side of \eqref{eq:graph-bound} contributes with the same power of $n$, heuristically indicating that this power is likely to be optimal.
  \end{enumerate}}
\end{remarks}

  The proof of the theorem is decomposed into several independent results. The variance lower bound   is established in the specific framework of Voronoi volume approximation. The Kolmogorov distance and moments upper bounds are potentially  valid in a more general framework.

\begin{theorem}
\label{thm:general-approx}
  Define $\sigma ^{2}=\var(\varphi (X))$. Assume that $\Vol(\partial K^{r})\leqslant S_+(K) r^{\alpha }$ for some $S_+(K) ,\alpha >0$. Then  (\ref{eq:volume-moments-upper-bound}) holds, and  for every $\varepsilon >0$ there is a constant $c_{\varepsilon }$ not depending on $n$ such that for $n\geq 1$,
{\begin{align}
\label{eq:general-tcl}
 d_{K}\left(\sigma ^{-1} ( {\varphi  (X)-\E \varphi (X))}  ,N\right)\leqslant c_{\varepsilon }\left( \sigma ^{-2}n^{-3/2 -\alpha /2d  } +\sigma ^{-3}n^{-2-\alpha / d } +\sigma ^{-4}n^{-3-\alpha /d}\right)  \log(n)^{3+\alpha /2d+\varepsilon }
 \end{align}} 
where $N$ is a standard Gaussian variable.

\end{theorem}

  Say that two points $x,y\in [0,1]^{d}$ are {\it Voronoi neighbours} among a point set $X$ if $V(x;X)\cap V(y;X)\neq \emptyset $. More generally, denote $d_{V}(x,y;X)$ the Voronoi distance between $x$ and $y $, i.e. the minimal integer $k\geq 1$ such that  we can form a path $x_{0}=x;x_{1}\in X,\dots ,x_{k-1}\in X,x_{k}=y$ where $x_{i}$ and $x_{i+1}$ are Voronoi neighbours. Denote $v(x,y;X)=\Vol(V(x,(X,y))\cap V(y,X))$ the volume that the cell $V(y,X)$ loses when $x$ is added to $X $. We have the explicit expression, for $x\notin X$,
\begin{align}
\label{eq:explicit-addone}
\varphi (X,x)-\varphi (X)=1_{\{x\in K\}}&\sum_{y\in X\cap K^{c}}v(x,y;X) -1_{\{x\in K^{c}\}}\sum_{y\in X\cap K}v(x,y;X).
\end{align}
Since $v(x,y;X)=0$ if $x$ and $y$ are not Voronoi neighbours in $(X,x,y)$, the concatenation of $X$ with $x$ and $y$,  the following properties hold.
\begin{proposition}Let $X=(X_{i})_{1\leqslant i\leqslant n}$ be  a finite collection of points.
\label{prop:phi-properties}\begin{enumerate}
 \item[{\rm (i)}] For $1\leqslant i\leqslant n$ such that $X_{i} \in K$ (resp. $ K^{c}$), if every Voronoi neighbour of $X_{i}$ among $X$ is also in $K$ (resp. $K^{c}$), then $D_{i}\varphi (X)=0$.
 \item[{\rm (ii)}]  For every point $X_{j} $ at Voronoi distance $>2$ from some $X_{i}\in X$, $D_{i,j}\varphi (X)=0$.
 
 \end{enumerate}
  \end{proposition}

\begin{remark}
These properties mean somehow that $\varphi $ is of range $2$ with respect to the Voronoi tessellation. An analogue of Theorem \ref{thm:general-approx}
should hold for any functional with finite range, such as the perimeter approximation induced by $\varphi _{\Per}$. On the other hand, the variance lower bound derived in this section is specific to the volume approximation.\end{remark}

We define for $x\in \mathbb{R}^{d},X=(X_{i})$ a finite collection of points, $ k\geq 1,$
\begin{align*}
R_{k}^{}(x;X)=\sup\{\|{y}-x\|:{ y\in V(X_{i};X)},d_{V}(x,X_{i};X )\leqslant k\}
\end{align*}the distance to the furthest { point in the cell of a } $k$-th order Voronoi neighbour, with $R(x;X):=R_{0}(x;X)$.  If $x$ does not have $k$-th order neighbours,  we put the convention $R_{k}^{}(x;X)=\text{diam}([0,1]^d)=\sqrt{d}$. { We have obviously}\begin{align}
\label{eq:volume-radius}
\Vol(V(x;X))\leqslant \kappa _{d}R(x;X)^{d}, x\in \mathbb{R}^{d},
\end{align}where $\kappa _{d}$ is the volume of the unit sphere in $\mathbb{R}^{d}$.

\begin{proof}[Proof of Theorem \ref{thm:general-approx}.]
We will use Theorem \ref{thm:dependancy} with the functional $f(X)=\varphi (X )-\E \varphi (X )$. 
Let us start with a crucial bound.

\begin{lemma}
\label{lm:variablesUkq}
Assume that (\ref{eq:volume-boundary-upper-bound}) holds.
Define for some $  k\geq 0,$ the random variable 
\begin{align*}
U_{k } =1_{\{d(  X_{1},\partial K)\leqslant R_{k}^{}({ X_{1}};X )\}}R_{k}^{}({  X_{1}};X )^{d}.
\end{align*}  Then for some $c_{d,qd+\alpha,k }>0$,
\begin{align*}
\E U_{k{ }} ^{q}\leqslant S_+(K) c_{d,qd+\alpha,k }n^{-q-\alpha/d },\;n\geq 1,q\geq 1.
\end{align*}
\end{lemma}

  \begin{proof} 
Under this form, it is problematic to give a sharp upper bound because the law of $R_{k}^{}(X_{1};X)$ depends on the position of $X_{1}$ within $[0,1]^d$. To inject some stationarity in the problem, we will bound $R_{k}^{}(X_{1};X)=R_{k}(X_{1};\hat X^{1})$ by introducing a closely related quantity $\overline{R}_{k}(X_{1};\hat X^{1})$ whose conditional law with respect to $\hat X^{1}$ is independent of the value of $X_{1}$. To this end, {\blu introduce the process
\begin{align*}
X'=\bigcup _{m\in \mathbb{Z} ^{d}}(X+m),
\end{align*}which law is invariant under translations. Remark that given any $t\in \mathbb{R}^{d}, X'$ has a.s. exactly $n$ points in $[t,t+1]^{d}$.  For $x\in \mathbb{R}^{d}$, call 
\begin{align*}
\mathcal{C}_{x}=\{[x-t,x-t+1]^{d};t\in [0,1]^{d}\}=\{[y,y+1]^{d}:y\in \mathbb{R}^{d},x\in [y,y+1]^{d}\},
\end{align*}the family of translates of $[0,1]^{d}$ that contain $x$. Then by stationarity of $X'$, the law $\mu _{k,n}$ of 
\begin{align*}
\overline{R_{k}}(x,X):=\sup_{C\in \mathcal{C}_{x}}R_{k}(x,X'\cap C)
\end{align*}does not depend on $x$ (and it is indeed only a function of $x$ and $X$). Also, for $x\in [0,1]^{d}$, $[0,1]^{d}\in \mathcal{C}_{x}$, whence $R_{k}(x,X)\leqslant \overline{R_{k}}(x,X)$.
}This yields \begin{align}
 \label{eq:EUkq}
 \E U_{k}^{q} &\leqslant   \int_{[0,1]^d}dx 1_{\{d(x;\partial K)\leqslant \overline{R}_{k}^{}(x;\hat X^{1})\}}\overline{R}_{k}^{}(x;\hat X^{1})^{qd}\\
  &\leqslant  \int_{\mathbb{R}_{+}\times [0,1]^d}1_{\{d(x,\partial K)\leqslant r\}}r^{qd}\mu _{k,n-1}(dr)dx \\
  &\leqslant  S_+(K)  \E \overline{R}_{k}^{}(0;\hat X^{1})^{qd+\alpha }\quad \text{using (\ref{eq:volume-boundary-upper-bound})}.
\end{align}  
{Let us now bound the probability of the event $\overline{R_{k}}(0,X)\geqslant r$, for some $r\geqslant 0$. If this event is realised, there is a $k$-th order Voronoi neighbour $z\in X'$ of $0$ and a point $y$ in the Voronoi cell of $z$ such that $\|y\|\geqslant r$. There is therefore a sequence of points $x_{1}=0,x_{2}\in X',\dots ,x_{k}=z,x_{k+1}=y$ such that for $i<k$, $x_{i}$ and $x_{i+1}$ are Voronoi neighbours. Since the midpoint $z_{i}$ of $x_{i}$ and $x_{i+1}$ has $x_{i}$ and $x_{i+1}$ as closest neighbours in $(X',0)$, the open ball $B^{o}(z_{i},\|x_{i}-x_{i+1}\|/2)  $ has an empty intersection with $X'$. Since $z$ is the point of $X'$ closest to $y$,  $B^{o}((z+y)/2,\|z-y\|/2)\cap X=\emptyset $ also. We therefore have  $k$ (possibly empty) open balls $B_{1},\dots ,B_{k}$, with respective radii $r_{i},i=1,\dots ,k$, such that $[x_{i},x_{i+1}]$ is a diameter of $B_{i}$, and such that $X'$ has a point in none of them. Since $\|y\|\geqslant r$, the radius of at least one of these balls is larger than $r/2k$. Define 
\begin{align*}
i_{0}:=\min\{1\leqslant i\leqslant k:r_{i}\geqslant r/2k\}.
\end{align*}We have by the triangular inequality $\|x_{i_{0}}\|\leqslant i_{0}r/2k\leqslant r/2$, and the ball $B(x_{i_{0}},r/2k)$ is empty of points of $X'$ and is contained in $[-r,r]^{d}$. It is easy to find  $\gamma _{d}>0$ such that at least one of the cubes $[g,g+\gamma _{d}r]^{d},g\in \gamma _{d} r\mathbb{Z} ^{d}\cap [-r,r]^{d}$ is contained in every ball with radius $r/2$ contained in $ [-r,r]^{d}$. This
yields
\begin{align*}
\P(\overline{R_{k}}(0,X)\geqslant r)&\leqslant \P(\exists g\in \gamma _{d} r\mathbb{Z} ^{d}\cap [-r,r]^{d}: X'\cap [g,g+\gamma _{d}r]^{d}=\emptyset )\\
&\leqslant \#(\gamma _{
d}'\mathbb{Z} ^{d}\cap [-1 ,1 ]^{d})\P([0,0+\gamma _{d}r]^{d}\cap X'=\emptyset ).
\end{align*}
Since $\#[0,0+\gamma _{d}r]^{d}\cap X'\geqslant n$ for $r\geqslant \gamma _{d}^{-1}$ and $X'\cap [0,0+\gamma_{d} r]=X\cap [0,0+\gamma _{d}r]$ for $r\leqslant \gamma _{d}^{-1}$, we finally have 
\begin{align*}
\P(\overline{R_{k}}(0,X)\geqslant r)&\leqslant 2^{d}\gamma _{d}^{-d}(1-\gamma^{d} _{d}r^{d})^{n}\leqslant 2^{d}\gamma _{d}^{-d}\exp(-n\gamma _{d}^{d}r^{d}).
\end{align*}
It then follows that for $u>0$,
\begin{align*}
\E \overline{R_{k}}(0,\hat X^{1})^{u}&=\int_{0}^{\infty }\P(\overline{R_{k}}(0,\hat X^{1})\geqslant r^{1/u})dr\leqslant  2^{d}\gamma _{d}^{-d}\int_{0}^{\infty }\exp(-(n-1)\gamma _{d}^{d}r^{d/u})dr \\
&\leqslant  2^{d}\gamma _{d}^{-d}(n-1)^{-u/d}\int_{0}^{\infty }\exp(-\gamma _{d}^{d}r^{d/u})dr.
\end{align*}
The conclusion follows by reporting this in \eqref{eq:EUkq}.}
 \end{proof}
  Proposition \ref{prop:phi-properties} and   (\ref{eq:volume-radius}) yield
for $q\geq 1$\begin{align*}
{  | \E D_{1}f(X)^{q}}   |  &\leqslant \kappa _{d}^{q} \E U_{1}^{qd}.
\end{align*} 
Lemma \ref{lm:variablesUkq} implies, for $q\geq 1$,
\begin{align}
\label{eq:Delta1-phi}
\E  | {D} _{1}f (X) | ^{q}\leqslant    c_{d,qd+\alpha }\kappa _{d}^{q}S_+(K)    n^{-q-\alpha/d },
\end{align} therefore the second term of the right-hand side of (\ref{eq:general-tcl}) follows immediately from the {\rosso last} estimate in (\ref{eq:graph-bound}). We now state Rhee-Talagrand's inequality \cite{RheTal86}, which then immediately yields (\ref{eq:volume-moments-upper-bound}).
\begin{lemma}[Rhee-Talagrand's inequality]
\label{lm:R-T}
Let $\psi  (X)$ be a symmetric measurable functional {\blu with finite $q$-th moment }. Then for $q\geq 1$ 
\begin{align*}
\E  | \psi  (X)-\E \psi  (X) |^{q} \leqslant n^{q/2}c_{q}\E D_{1} | \psi (X) |  ^{q}
\end{align*}  with $c_{q}=2^q(18\sqrt{q}q')^{q'}$, where $1/q+1/q'=1$. For $q=2$, Stein-Efron's inequality yields the better   constant $c_{2}=1/2$.
\end{lemma}

{
Let us bound the two first terms of \eqref{eq:graph-bound}.
We need for that to control  the maximum radius of Voronoi cells over $X$. 
We first introduce the event on the circumscribed radii of the Voronoi spheres, 
\begin{equation*}
\Omega _{n}(X)=\left( \max_{1\leqslant j\leqslant n}(R(X_{j};X ))\leqslant n^{-1/d}\rho_{n} \right)
\end{equation*}     where $\rho_{n}=\log(n)^{1/d+\varepsilon '}$ for $\varepsilon '$ sufficiently small.
We have the following lemma, proved later for the sake of readability.
\begin{lemma}\label{lm:omega_{n}}For all $\eta >0,$ $n^{\eta }\P(\Omega _{n}(X)^{c})\to 0$ as $n\to \infty $. 
\end{lemma}

To bound the first term of  \eqref{eq:graph-bound}, {let   $Y,Y',Z$ be recombinations of $\{X,X',\tilde X\}$. Introduce the event $\Omega :=\Omega _{n}(Y)\cap \Omega _{n}(Y')\cap \Omega_{n}(Z)\cap \Omega _{n}(Z') $ which satisfies $\P(\Omega ^{c})\leqslant 4\P(\Omega _{n}(X)^{c})$.} Recall the fact that $D_{ij}f(X)$ can only be non-zero if $X_{j}$ is at Voronoi distance $\leqslant 2$ from $X_{i}$, and that $D_{j}f(X)$ can only be non-zero if $X_{j}$   has a Voronoi neighbour which cell touches $\partial K$. { In the notation of \eqref{eq:graph-bound}, we have\begin{align*}
 \E  1_{\{D_{1,2}\varphi (Y)\neq 0\}}D_{1}\varphi (Z)^{4} &\leqslant \E 1_{\Omega  }1_{\{D_{1,2}\varphi (Y)\neq 0\}}D_{1}\varphi (Z)^{4}+\P(\Omega ^{c})\\
&\leqslant \kappa _{d}^{4}n^{-4}\rho_{n}^{4d}\E [1_{\{d(Y_{1},\partial K)\leqslant 2n^{-1/d}\rho_{n}\}}\E[1_{\{\|Y_{1}-Y_{2}\|\leqslant 2n^{-1/d}\rho_{n}\}} |  Y_{1} ]]+\P(\Omega ^{c})\\
&\leqslant \kappa _{d}^{5}n^{-4}\rho_{n}^{4d}2^{d}n^{-1}\rho_{n}^{d}\P(d(Y_{1},\partial K)\leqslant 2n^{-1/d}\rho_{n})+\P(\Omega ^{c})\\
&\leqslant C_{1,2}n^{-5-\alpha /d}\rho_{n}^{5d+\alpha }
\end{align*}for some $C_{1,2}\geq 0$, whence {\blu Proposition} \ref{r:zee} and \eqref{eq:coef-CS} yield {\blu $n \coef\leqslant C'n^{-4-\alpha /d}\rho _{n}^{5d+\alpha } $ for some $C'>0$}. With a similar computation,
\begin{align*}
 \E \mathbf{1}_{\{\Omega \}}&\mathbf{1}_{\{D_{1,2}\varphi (Y)\neq 0,D_{1,3}\varphi (Y')\neq 0\}}D_{2}\varphi (Z)^{4}\\
&\leqslant \kappa _{d}^{4}n^{-4}\rho _{n}^{4d} \P({\|Y_{1}-Y_{2}\|\leqslant 2n^{-1/d}\rho_{n},\|Y'_{1}-Y'_{3}\|\leqslant 2n^{-1/d}\rho_{n}},{d(Y_{1},\partial K)\leqslant 2n^{-1/d}\rho_{n}})+\P(\Omega ^{c})\\
&\leqslant C_{2,3}n^{-6-\alpha /d}\rho_{n}^{6d+\alpha },
\end{align*}from where {\blu $n^{2}\coeff\leqslant C''n^{-4-\alpha /d}\rho _{n}^{6d+\alpha } $ for some $C''>0$}. Therefore the first term of \eqref{eq:graph-bound} is bounded by 
\begin{align*}
\sigma ^{-2}\sqrt{n}(n^{-2-\alpha /2d})\log(n)^{3+\alpha /2d+d\varepsilon' /2}
\end{align*}up to a constant, which yields the first term of \eqref{eq:general-tcl}.} It remains to bound the term
\begin{align*}
\E  | f(X) |  | D _{j}f(X^{A} ) | ^{3}
\end{align*}from (\ref{eq:graph-bound}). 
 Recall that under $\Omega _{n}(X^{A})$, all Voronoi cells volumes, and therefore all $ | D _{j}f(X^{A}) | $,{\blu $1\leqslant j\leqslant n$}, are bounded by $\kappa _{d}n^{-1}\rho_{n}^{d}$, and also, $D_{j}f(X^{A})=0$  if $X_{j}$ and $X_{j}'$ are  at distance more than $2n^{-1/d}\rho_{n}$ from $K's$ boundary. We have  
\begin{align*}
\E  | f(X)D _{j}f(X^{A}) | ^{3}&\leqslant \E \left(  | f(X) |  | D _{j}f(X^{A}) |^{3}1_{\Omega _{n}}(X^{A}) \right)+\P(\Omega _{n}(X)^{c})\\
 &\leqslant cn^{-3  }\rho_{n}^{3d} \E \left[  | f(X) | 1_{\{X_{j}\text{ or }X_{j}'\in \partial  K^{2n^{-1/d  }\rho_{n}}\}} \right]+\P(\Omega _{n}(X)^{c})\\
 &\leqslant cn^{-3  }\rho_{n}^{3d}\E \left( \left(  | f(\hat X^{j}) | + | D_{j}f(X) |  \right)1_{\left\{X_{j}\text{ or }X_{j}'\in \partial  K^{2n^{-1/d }\rho_{n}}\right\}}  \right)+\P(\Omega _{n}(X)^{c}) .
 \end{align*}  We have 
\begin{align*}
\E  | D_{j}f(X)|  \leqslant c' n^{-1-\alpha /d}
\end{align*}by (\ref{eq:Delta1-phi}), while the other term is bounded by independence by
\begin{align*}
\E  | f(\hat X^{j}) | 1_{\left\{X_{j}\text{ or }X_{j}'\in \partial  K^{2n^{-1/d }\log(n)}\right\}}&\leqslant 2 \E  | f(\hat X^{j}) | \P\left(   {X_{j}\in \partial  K^{2n^{-1/d  }\rho_{n}}}  \right) \\
&\leqslant c''\sigma  n^{-\alpha /d }\rho_{n}^{\alpha }.
\end{align*}Finally, for some $C>0$,
\begin{align*}
\E  | f(X)D _{j}f(X^{A}) | ^{3}\leqslant Cn^{-3- \alpha /d  }\log(n)^{3+\varepsilon /2}(\sigma  \log(n)^{\alpha/d+\varepsilon /2 }+n^{-1}),
\end{align*} 
which gives the desired bound.
}
 
\end{proof}

  \begin{proof}[Proof of Lemma \ref{lm:omega_{n}}]
  
  We can find a constant $\gamma _{d}>0$ such that the intersection with $[0,1]^{d}$ of every ball centred in $[0,1]^{d}$ of radius ${\blu r\leqslant 1}$ contains a cube $g+[0,\gamma _{d}{\blu r}]^{d}$ for some $g\in \gamma _{d}{\blu r}\mathbb{Z} ^{d}$. If $\max_{1\leqslant j\leqslant n} R(X_{j};X)>n^{-1/d}\rho _{n}$, then two Voronoi neighbours $X_{i},X_{j}$ are at distance more than $n^{-1/d}\rho _{n}$ from one another, and the open ball with diameter $[X_{i},X_{j}]$ does not contain points of $X$, by the construction of  the Voronoi tessellation. It follows that a cube $g+[0,\gamma _{d}n^{-1/d}\rho _{n}]^{d}\subseteq [0,1]^{d}$ is empty of points of $X$, for some $g\in \gamma _{d}n^{-1/d}\rho _{n}\mathbb{Z} ^{d}$, and this event  happens with a probability bounded by 
\begin{align*}
(\gamma _{d}n^{-1/d}\rho _{n})^{-d}\P([0,\gamma _{d}n^{-1/d}\rho _{n}]^{d}\cap X=\emptyset )&\leqslant \gamma _{d}^{-d}n\rho _{n}^{-d}(1-\gamma _{d}^{d}n^{-1}\rho _{n}^{d})^{n}\\
  &\leqslant \gamma _{d}^{-d}n\rho _{n}^{-d}\exp(n\log(1-\gamma _{d}^{d}n^{-1}\rho _{n}^{d}))\\
  &\leqslant \gamma _{d}^{-d}n\rho _{n}^{-d}\exp(- \gamma _{d}^{d} \log(n)^{1+d\varepsilon '}),
\end{align*}
which proves the result.\end{proof}

 \begin{proof}[Proof of Theorem \ref{thm:voronoi-tcl}]

It only remains to prove the lower bound on the variance in (\ref{eq:hyp-lower-bound-boundary}).
Lemma \ref{lem:var-lower-bound}  states that the variance is larger than $n\|h\|_{L^{2}([0,1]^d,\ell)}^{2}$, where 
\begin{align*}
h(x)=\E \varphi  (\hat X^{1},x)-\E\varphi  ( X),\quad x\in [0,1]^d.
\end{align*}
We decompose $h$ as follows:
\begin{align}
\label{eq:somme}
\nonumber h(x)&=(\E \varphi  (\hat X^{1},x)-\varphi  (\hat X^{1})-(\E\varphi  (X)-\varphi  (\hat X^{1} )), x\in [0,1]^d\\
&=:h_{1}(x)-h_{2}.
\end{align}
Voronoi volume approximation is not homogeneous in the sense that points falling close to $K$'s boundary have more influence than other points of $X_{n}$. The following lemma shows that this inhomogeneity makes $h_{1}$ the dominant term in the previous decomposition.

\begin{lemma}
\label{lm:variance-lower-bound}Let $K$ be a measurable subset of $[0,1]^d$, define $h_{1}$ as in (\ref{eq:somme}). Then we have\begin{align*}
\int_{[0,1]^{d}}h_{1}(x)^{2}dx\geq C_{d} (\gamma  (K,n^{-1/d})+\gamma (K^{c},n^{-1/d}))n^{-2}  
\end{align*}for some $C_{d}>0$.
\end{lemma}
Let us first conclude the proof of Theorem \ref{thm:voronoi-tcl}. If the weak  rolling ball condition is satisfied along with  (\ref{eq:hyp-lower-bound-boundary}), it yields  \begin{align*}
\int_{[0,1]^d}h_{1}(x)^{2}dx&\geq C_{d} S_-(K) \gamma (K)(n^{-1/d})^{\alpha }n^{-2}.
\end{align*}
According to Lemma \ref{lm:variablesUkq}, $h_{2} =  O(n^{-1-\alpha /d})$, which is indeed negligible with respect to $\|h_{1}\|_{L^{2}}\geq C_{d,K}n^{-1-\alpha /2d}$.
 \end{proof}
\begin{proof}[Proof of Lemma \ref{lm:variance-lower-bound}]
 It follows from (\ref{eq:explicit-addone}) that for $x\in K^{c}$ 
 \begin{align*}
 | \varphi  (x,\hat X^{1})-\varphi  (\hat X^{1}) | &=\sum_{j=2}^{n} 1_{ \{X_{j}\in K\}}v(x,X_{j};\hat X^{1}), \end{align*}  where we notice that the summand distribution does not depend on $j$. Then \begin{align*}
 | h_{1}(x) | &\geq 1\left(x\in  {\partial K_{+} ^{n^{-1/d}}} \right)(n-1)\E  1_{\{X_{2}\in K\}}v(x,X_{2};\hat X^{1})\\
 & \geq 1\left(x\in  {\partial K_{+} ^{ n^{-1/d}}}  \right)(n-1)   \E\int_{y\in  K } v(x,y;\hat X^{1,2})dy \\ 
 &\geq  1\left(x\in  {\partial K_{+} ^{n^{-1/d}}}  \right)(n-1)\Vol( B {(x,\beta n^{-1/d})}\cap K) \inf_{y:\|y-x\|\leqslant \beta  n^{-1/d}  }\E v(x,y;\hat X^{1,2}).\\
 \end{align*}

 If for some $y\in [0,1]^{d},\varepsilon >0$, no point of $\hat X^{1,2}:=(X_{i})_{i\neq 1,2}$ falls in $B(y,6\varepsilon )$, then $B(y,3\varepsilon )\subset V(y,\hat X^{1,2})$. If furthermore $x\in [0,1]^{d}$ lies at distance less than $\varepsilon $ from $y$, then with $z=x+\varepsilon \|x-y\|^{-1}(x-y)$, $$B(z,\varepsilon )\subset V(x,(\hat X^{1,2}, y))\subset B(y,3\varepsilon )\subset V(y; \hat X^{1,2}),$$
and therefore $v(x,y;\hat X^{1,2})\geq \kappa _{d}\varepsilon ^{d}$. We finally have 
\begin{align*}
 \inf_{y:\|y-x\|\leqslant \beta n^{-1/d}  }\E v(x,y;\hat X^{1,2})\geq \kappa _{d}\beta ^{d}n^{-1}\P(\hat X^{1,2}\cap B(y,6\beta n^{-1/d} )=\emptyset )\geq c_{d}'n^{-1}
\end{align*}for some $c_{d}'>0$.
With a completely similar result for $x\in K$, we have for some $c_{d}''>0$
\begin{align*}
\int_{W}h_{1}(x)^{2}dx&\geq  c_{d}^{''} \left( \int_{\partial K_{+}^{n^{-1/d}}}\Vol(B(x,\beta n^{-1/d})\cap K) ^{2}dx+\int_{\partial K_{-}^{n^{-1/d}}}\Vol(B(x,\beta n^{-1/d})\cap K^{c}) ^{2}dx \right).
\end{align*}

\begin{remark}
\label{rk:optimal}
All three terms of (\ref{eq:graph-bound}) give in the case of Theorem \ref{thm:voronoi-tcl} a bound of order $n^{-1/2+\alpha /2d }\log(n)^{q}$ for some $q>0$. In these conditions it seems hard to reach a Berry-Essen bound negligible with a better magnitude than $n^{-1/2+\alpha /2d}$, but removing the $\log $ is an open problem.
\end{remark}

\end{proof}

\subsection{Covering processes}
\label{sec:covering}
\newcommand{\F}{\mathscr{F}}
\newcommand{\K}{\mathcal{K}}

\newcommand{\sK}{\mathscr{K}}
Let  $(\K,\sK)$ be the space of  compact subsets of $\mathbb{R}^{d}$, endowed with the hit-and-miss topology and a Borel  probability measure $\nu $. Let $E_{n}$ be a cube of volume $n$, and $C_{1},\dots ,C_{n}$ iid uniform variables in $E_{n}$, called the \emph{germs}.
Let $n$ iid compact sets $K_{1},\dots ,K_{n}$ be distributed as $\nu $, called the \emph{grains},  and define the \emph{germ-grain process} $X_{i}=C_{i}+K_{i}$, $i=1,\dots ,n$. { An important feature of the model regarding Gaussian approximation is the radius
\begin{align*}
R_{i}:=\sup\{\|x\|:x\in K_{i}\}, 1\leqslant i\leqslant n.
\end{align*} }We consider the random closed set formed by the union of the grains translated by the germs
\begin{align*}
F_{n}=\left( \cup _{k=1}^{n}X_{k} \right)\cap E_{n}.
\end{align*}We are interested in the volume of $C_{n}$ covered by $F_{n}$
\begin{align*}
f_{V}(X_{1},\dots ,X_{n})=\Vol (F_{n}),
\end{align*} the number  of isolated grains 
\begin{align*}
f_{I}(X_{1},\dots ,X_{n})=\#\{k:  X_{k}\cap X_{j}\cap E_{n}=\emptyset, k\neq j \},
\end{align*} and their centred versions with unit variance $\tilde f_{V},\tilde f_{I}$.  The functional $f_{V}$ denotes the total volume of the germ-grain process, and $n^{-1}f_{V}(X_{1},\dots ,X_{n})$ can serve as an estimator for the \emph{fraction volume}, i.e. the portion of the space  occupied by the boolean model $\cup _{k}X_{k}$, and therefore be used in estimating the parameters of $\nu $ (see \cite{Mol97} for insights on the boolean model statistics).

Kolmogorov Berry-Essen bounds in $n^{-1/2}$ for binomial input  for $f_{V}$  or $f_{I}$ have only been obtained very recently in  \cite{GolPen10} with balls with deterministic identical radii (with the possibility to extend the method to a random radius), using size-biased couplings. Chatterjee \cite{Cha08} obtained similar bounds in Wasserstein distance. We present here the first such bounds in the unbounded random grain context. Furthermore, the computations are quite straightforward and the method is generalisable to similar local functionals of the boolean model, such as the perimeter, or other Minkowski functionals. The use of the bound \eqref{eq:graph-bound}  is crucial to have  a decay in $n^{-1/2}$ { in the context of random grains}.
The variance is a straightforward computation of integral geometry, it is a consequence of for instance \cite[Th. 4.4]{KenMol} that under the conditions of the theorem below, we have  $cn\leqslant \var f(X_{1},\dots ,X_{n}) \leqslant  Cn$ for some $c,C>0$, for $f=f_{V}$ or $f=f_{I}$.

\begin{theorem}
Assume that 
{
$\E R_{1}^{5d}<\infty .
 $} Let $N$ be a standard Gaussian variable.
Then we have for some $C>0$,
\begin{align*}
d_{K}&(\tilde f_{V}(X_{1},\dots ,X_{n}),N)\leqslant Cn^{-1/2}.
\end{align*}{If $\E R_{1}^{8d}<\infty $,} for some $C'>0$,
\begin{align*}
d_{K}&(\tilde f_{I}(X_{1},\dots ,X_{n}),N)\leqslant C'n^{-1/2}.
\end{align*}
\end{theorem}

\begin{proof}  Let first $f=f_{V}$. Given a $n$-tuple $x=(x_{1},\dots ,x_{n})\in  \K^{n}$, we have $D_{i,j}f(x)= 0$ as soon as
 $
 \Vol(x_{i}\cap x_{j})= 0,
 $
which gives us a sufficient condition. { Let us estimate the right hand side of (\ref{eq:graph-bound}). Introduce independent copies $X',\tilde X$ of $X$, and for $U$ a random compact set among those families, denote by $c(U)$, $ r(U), K(U)$ its centre, radius, and grain, so that $$\{c(X_{i}),c(X'_{i}),c(\tilde X_{i}),K(X_{i}),K(X'_{i}),K(\tilde X_{i}),1\leqslant i\leqslant n\}$$ is a family of independent variables. Let us write $V_{i}=\Vol(X_{i}),V_{i}'=\Vol(X_{i}')$. We have  $ | D_{1}f_{V}(X) | \leqslant V_{1}$, and  since  the volume has a finite moment of order $5$,   \begin{align*}
\sup_{n\geq 1} \E  | D_{1}f(X) | ^{3}  <\infty ,\;\sup_{n\geq 1} \E  | D_{1}f(X) | ^{4}  <\infty .
\end{align*}
We also have for $A\subseteq [n]$
\begin{align*}
\E  | f(X) | | D _{j}f(X^{A}) |^{3}&\leqslant \E  | f(X^{\hat j}) D _{j}f(X^{A}) | ^{3}+\E  | D_{j}f(X)D_{j}f(X^{A})^{3} |\\
&\leqslant      \E  | f(X^{\hat j}) | (V_{j}^{3}+(V_{j}')^{3})+\E D_{j}f(X)^{4}\\
&\leqslant \E  | f(X^{\hat j}) |2\E V_{j}^{3}+\E V_{j}(X)^{4}, 
\end{align*}whence 
\begin{align*}
\sigma ^{-4}n\E  | f(X)   D _{j}f(X^{A})   ^{3}  | \leqslant Cn^{-1/2}
\end{align*}for some $C>0$.

  To estimate $\coef,\coeff$, we use {\blu Proposition} \ref{r:zee}, \eqref{eq:coeff-CS}, and \eqref{eq:coef-CS}. Fix $Y,Y',Z$ recombinations of $\{X,X',\tilde X\}$,  we have
\begin{align*}
 \E \mathbf{1}_{\{D_{1,2}f  (Y)\neq 0 \}}D_{1}f (Z)^{4}&\leqslant \E\mathbf{1}_{\{Y_{2}\cap Y_{1}\neq \emptyset\}} \Vol(Z_{1})^{4}\\
&\leqslant \E\left[  \kappa _{d}^{4}r(Z_{1})^{4d}\P(c(Y_{2})\in B(c(Y_{1}),r(Y_{1})+r(Y_{2})) |  Y_{1},Z_{1} ,r(Y_{2})) \right]\\
&\leqslant n^{-1}\kappa _{d}^{5}\E r(Z_{1})^{4d}(r(Y_{1})+r(Y_{2}))^{d} 
\end{align*}whence $\sup_{n}n\coef<\infty $ since $\E R_{1}^{5d}<\infty $.

Then,
\begin{align*}
\E &\mathbf{1}_{\{D_{1,2}f  (Y)\neq 0,D_{1,3}f  (Y')\neq 0 \}}D_{2}f (Z)^{4}\\
&\leqslant \E \left[ \Vol(Z_{2})^{4}\mathbf{1}_{\{D_{12}f(Y)\neq 0\}}\P(c(Y'_{3})\in B(c(Y'_{1}),r(Y'_{1})+r(Y'_{3})) |  Z_{2},Y_{1},Y_{2},Y'_{1} ,r(Y'_{3})) \right]\\
&\leqslant n^{-1}\kappa _{d}^{5}\E\left[  r(Z_{2})^{4}(r(Y'_{1})+r(Y'_{3}))^{d}\P(c(Y_{2})\in B(c(Y_{1}),r(Y_{1})+r(Y_{2})) |  Z_{2},Y_{1},Y'_{1},Y'_{3},r(Y_{2}) ) \right]\\
&\leqslant n^{-2}\kappa _{d}^{6}\E r(Z_{2})^{4}(r(Y'_{1})+r(Y'_{3}))^{d}(r(Y_{1})+r(Y_{2}))^{d}.
\end{align*}Using the definition of recombinations, the variables $Y_{1}',Z_{2},Y'_{3}$ are pairwise independent,  and the expectation above is finite because of $\E r(X_{1})^{5d}<\infty $.
We indeed have $\sup_{n} n^{2}\coeff<\infty $, which concludes the proof for the Kolmogorov bound on $\tilde f_{V}$. \\

Dealing with $f=f_{I}$ is slightly more complicated. Introduce $d_{i,j}(X)$ the distance between $i$ and $j$ in the germ-grain process $X$, defined as the smallest number $q$ such that there is a chain $i_{1}=i,\dots ,i_{q}=j$ such that $X_{i_{k}}\cap {X_{i_{k+1}}}\neq \emptyset $. Call $B_{i}^{p}(X)$ the number of points at distance $\leqslant p$ from the point $i$ for the distance $d_{\cdot ,\cdot }(X)$. { For some $1\leqslant i,j\leqslant n,$ the value of the functional 
\begin{align*}
\mathbf{1}_{\{X_{j}\text{ is isolated}\}}:=\mathbf{1}_{\{X_{j}\cap X_{k}\cap E_{n}=\emptyset ,k\neq j\}}
\end{align*}can be affected by the removal of $X_{i}$ only if $X_{i}\cap X_{j}\neq \emptyset $, therefore, for $1\leqslant i\leqslant n$,
\begin{align*}
 | D_{i}f _{I}(X) | \leqslant \# B_{i}^{1}(X),
\end{align*}}whence,
\begin{align}
\label{eq:bound-4thderivative-isolated}
\E | D_{1}f_{I}(X)  | ^{q}&\leqslant \E\# B_{i}^{1}(X)^{q}, q\leqslant 1.
\end{align}We will estimate this bound later.
With the same notation than for the functional $f_{V}$, let us now deal with $\coef,\coeff$. Remark that $D_{i,j}f_{I}(X)=0$ if $d_{i,j}(X)>2$. 
 We have 
\begin{align*}
\coef&\leqslant \sup_{(Y,Z)} \E \mathbf{1}_{\{2\in B_{1}^{2}(Y)\}}\#B_{1}^{1}(Z)^{4}
\end{align*}{  and 
\begin{align*}
\mathbf{1}_{\{2\in B_{1}^{2}(Y)\}}\leqslant \sum_{k}\mathbf{1}_{\{X_{1}\cap X_{k}\neq \emptyset ,X_{2}\cap X_{k}\neq \emptyset \}}.
\end{align*}}
To simplify notation, remark that for $Y,Z$ recombinations of $\{X,X',\tilde X\}$, $\#B_{1}^{p}(Y)\leqslant \# B_{1}^{p}(T)$, where $T$ is the concatenation of $Y$ and $Z$ and is in fact composed of $m$ iid variables distributed as $X_{1}$, where $n\leqslant m\leqslant 2n$. We then have
\renewcommand{\k}{\mathbf{k}}
\begin{align}
\label{eq:bound-coef-isolated}
\coef&\leqslant\sup_{n\leqslant m\leqslant 2n}\E  \sum_{k=1}^{m}\mathbf{1}_{\{T_{1}\cap T_{k}\neq \emptyset ,T_{k}\cap T_{2}\neq \emptyset \}}\sum_{1\leqslant k_{1},k_{2},k_{3},k_{4}\leqslant m}\mathbf{1}_{\{T_{k_{i}}\cap T_{1}\neq \emptyset,i=1,\dots ,4 \}},
\end{align} and the supremum is reached for $m=2n$.
We have similarly, with $m=3n$,
\begin{align}
\label{eq:bound-coeff-isolated}
\coeff &\leqslant  \E \sum_{k=1}^{m}\mathbf{1}_{\{T_{1}\cap T_{k}\neq \emptyset ,T_{2}\cap T_{k}\neq \emptyset \}}  \sum_{k'=1}^{m}\mathbf{1}_{\{T_{1}\cap T_{k'}\neq \emptyset ,T_{3}\cap T_{k'}\neq \emptyset \}}  \sum_{\mathbf{k}=(k_{1},k_{2},k_{3},k_{4})\in [m]^{4}}\mathbf{1}_{\{T_{1}\cap T_{k_{i}}\neq \emptyset \}}.
\end{align}
To estimate   \eqref{eq:bound-4thderivative-isolated}-\eqref{eq:bound-coeff-isolated}, it is useful to introduce some more notation. Call graph on $[n]$ the finite data of distinct edges $t=\{\{i_{1},j_{1}\},\dots ,\{i_{q},j_{q}\}\}$. For such a graph, introduce the probability
\begin{align*}
p(t)=\P(T_{i_{1}}\cap T_{j_{1}}\neq \emptyset ,\dots ,T_{i_{q}}\cap T_{j_{q}}{\rosso \neq }\emptyset ).
\end{align*}
Say that this graph is a tree when it is connected and  has no cycles.
Let us prove that  for every tree $t$ with $q$ distinct vertices,  
\begin{align}
\label{eq:grauph-bound-isolated}
p(t)\leqslant  (d\kappa _{d}n^{-1})^{q-1}\E r(T_{1})^{(q-1)d}.
\end{align}

Let $t$ be such a tree, and let an arbitrary vertex $i_{0}$ of $t$, designated to be the root of $t$.  Call $\mathcal{G}_{k}(t),k\geq 1,$ the members of the $k$-th generation, noticing that there can not be more than $q$ generations, i.e. $\mathcal{G}_{k}(t)=\emptyset $ for $k>q$. Call $\mathcal{G}_{k}^{-}(t)=\cup _{j<k}\mathcal{G}_{j}(t),\mathcal{G}_{k}^{+}(t)=\mathcal{G}_{k}(t)\setminus \mathcal{G}_{k}^{-}(t)$, and call $\mathcal{G}_{k}^{k+1}(t)$ the { collection of all pairs} $(i,j)$ such that $i\in \mathcal{G}_{k}(t),j\in \mathcal{G}_{k+1}(t),\{i,j\}\in t$.
We have 
\begin{align*}
p(t)&\leqslant \E\left[ \mathbf{1}_{\{T_{i}\cap T_{j}\neq \emptyset ;\{i,j\}\in t;i,j\in\mathcal{G}_{q}^{-}(t) \}}\right.\\
&\hspace{1.5cm}\left. \P\left(c(T_{j})\in B(c(T_{i}),r(T_{i})+r(T_{j}));(i ,j)\in \mathcal{G}^{q}_{q-1}(t) \;\Big|\;  c(T_{i}),i\in \mathcal{G}_{q}^{-}(t){\red ;}\;r(T_{i}),i\in [m]\right) \right]\\
&\leqslant \E\left[ \mathbf{1}_{\{T_{i}\cap T_{j}\neq \emptyset ;\{i,j\}\in t,i,j\in\mathcal{G}_{q}^{-}(t) \}}  \prod_{(i,j)\in \mathcal{G}_{q-1}^{q}(t)}n^{-1}\kappa _{d}(r(T_{i}{\rosso )}+r(T_{j}))^{d}\right]\\
&\leqslant (\kappa _{d}n^{-1})^{\#\mathcal{G}_{q-1}^{q}(t)}\E\left[   \mathbf{1}_{\{T_{i}\cap T_{j}\neq \emptyset ;\{i,j\}\in t,i,j\in\mathcal{G}_{q}^{-}(t) \}}\prod_{ {\rosso\{i,j\}}\in t:i,j\in  \mathcal{G}_{q}^{+}(t)}(r(T_{i})+r(T_{i}))^{d}
 \right].\end{align*}
  Applying this procedure inductively back until  the $1$-st generation , that is the root $i_{0}$ of the tree, yields
  
\begin{align*}
p(t)&\leqslant (\kappa _{d}n^{-1})^{\sum_{k\geq 1}\#\mathcal{G}_{k}^{k+1}(t)}\E \left[ \prod_{\ (i,j)\in\cup _{k} \mathcal{G}_{k}^{k+1}(t)}(r(T_{i})+r(T_{j}))^{d} \right].
\end{align*}Now, $\cup _{k\geq 1}\mathcal{G}_{k}^{k+1}(t)$, contains all the ${\rosso  q-1}$ edges of $t$, whence 
\begin{align*}
p(t)\leqslant \kappa _{d}^{q-1}n^{-(q-1)}\E \prod_{\{i,j\}\in t}(r(T_{i})+r(T_{j}))^{d}\leqslant (d\kappa _{d}n^{-1})^{q-1} \E r(T_{1})^{(q-1)d}, 
\end{align*}by using Cauchy-Schwarz inequality, whence \eqref{eq:grauph-bound-isolated} follows.
  
  We have 
\begin{align*}
\E  | D_{1}f_{I}(X) | ^{6}\leqslant \sum_{\k=(k_{1},\dots ,k_{6})\in [m]^{6}}p(\{1,k_{i}\},i=1,\dots ,6)\leqslant C n^{-5}
\end{align*} for some $C>0$, by using $\E r(X_{1})^{5d}<\infty $, which treats all the terms of \eqref{eq:graph-bound} except the ones containing $\coef$ and $\coeff$.

  We call, for $u_{1},\dots ,u_{q}$ distinct integers, $l\geq 0,p\geq 4$,  
  \begin{align*}
[m]_{u_{1},\dots ,u_{q};l}^{p}=\{\k=(k_{1},\dots ,k_{p})\in [m]^{p}\;:\;\#\{u_{1},\dots ,u_{q},k_{1},\dots ,k_{p}\}\}=q+l.
\end{align*}We can easily prove that there are constants $C_{l}$ {\rosso not depending on $m$} such that
\begin{align}
\label{eq:ml-estimates}
 \#[m]^{p}_{u_{1},\dots ,u_{q};l}\leqslant C_{l}n^{l}.
\end{align}
We have, for $T$ with $2n$ iid components, using \eqref{eq:bound-coef-isolated},
\begin{align*}
\coef& \leqslant \sum_{k=1}^{n}\sum_{\k=(k_{i})\in [2n]^{4}}p(\{1,k\},\{2,k\},\{1,k_{i}\};i=1,\dots ,4)\\
& \leqslant   \sum_{l=0}^{5}\sum_{\k\in [m]^{5}_{1,2;l}}p(\{1,k_{1}\},\{2,k_{1}\},\{1,k_{i}\};i=2,\dots ,5).
\end{align*}For $\k\in [m]^{5}_{1,2; l}$,  one can easily extract a tree with $l+1$ edges from $\{\{1,k_{1}\},\{2,k_{1}\},\{1,k_{i}\};i=2,\dots ,5\}$, whence \eqref{eq:grauph-bound-isolated}  yields 
\begin{align*}
\coef\leqslant C \sum_{l=0}^{5}\sum_{\k\in [m]^{5}_{1,2;l}}n^{-l-1}\leqslant C'n^{-1},
\end{align*} using also \eqref{eq:ml-estimates}. This gives ${\rosso \sup_{n}n\coef<\infty }$. Similar computations yield
\begin{align*}
\coeff &\leqslant \E \sum_{k}\mathbf{1}_{\{T_{1}\cap T_{k}\neq \emptyset ,T_{2}\cap T_{k}\neq \emptyset \}}  \sum_{k'}\mathbf{1}_{\{T_{1}\cap T_{k'}\neq \emptyset ,T_{3}\cap T_{k'}\neq \emptyset \}}  \sum_{\mathbf{k}=(k_{1},k_{2},k_{3},k_{4})\in [m]^{4}}\mathbf{1}_{\{T_{1}\cap T_{k_{i}}\neq \emptyset \}}\\
 &\leqslant \sum_{ \mathbf{k}=(k_{i})\in [m]^{6}}p(\{1,k_{1}\},\{2,k_{1}\},\{1,k_{2}\},\{3,k_{2}\},\{1,k_{i}\}, i=3,\dots ,6) \\ 
    &= \sum_{l=0}^{6} \sum_{ \mathbf{k}=(k_{i})\in [m]_{1,2,3;l}^{6}}p(\{2,k_{1}\}, \{3,k_{2}\},\{1,k_{i}\}, i=1,\dots ,6) 
\end{align*}
and for $\mathbf{k}\in [m]^{6}_{1,2,3;l}$ one can extract a tree with $l+2$ edges from $\{\{2,k_{1}\},\{3,k_{2}\},\{1,k_{i}\};i=1...6\}$, whence
\begin{align*} B_{n}'(f)\leqslant \sum_{l=0}^{6} \sum_{\k\in [m]^{6}_{1,2,3;l}}(\kappa _{d}dn^{-1})^{l+2}
\leqslant Cn^{-2},
\end{align*}
which concludes the proof.

   }
\end{proof}
 
\subsection{Further applications} 
 
 It is proved in \cite{Cha08} that, in the notation of
  Theorem \ref{thm:SchulteSchatterjeebound} and for $\sigma =1$,
\begin{align}
\label{eq:wasserstein-stein}d_{W}(W,N) &\leqslant  \delta _{1}+\delta _{2}\\  \delta _{1}:&=\sqrt{\var(\E(T | X))}\\
\delta _{2}:&=2c\sum_{j=1}^{n}\E  | \Delta _{j}f(X) | ^{3}
\end{align}where $d_W$ is the 1-Wasserstein distance. This bound has been successfully applied in \cite{Cha08}, \cite{Cha13}, and \cite{Nol} to several normal approximation problems. Without fully developing the details, we indicate here how we can obtain similar bounds in the Kolmogorov's distance by using the techniques developed in this paper.  Assuming that $\sigma =1$, the new terms in (\ref{eq:abstract-bound}) with respect to (\ref{eq:wasserstein-stein}) are 
\begin{align*}
 \delta _{1}'&=\sqrt{\var(\E(T' | X))}\\ \delta _{2}'&=6  \sum_{j=1}^{n}\sqrt{\E  |D_{j}f(X) | ^{6}}.
\end{align*} The term $\delta _{1}'$ is very close in its expression to $\delta _{1} $.  In the  examples developed below, it is indeed possible to apply the bound  already derived for $\delta _{1} $ to $\delta _{1}'$.  The term $\delta _{2}'$ has to be dealt with separately, it is in general more straightforward. Remark that $\delta _{2}'$ can be replaced by the bound $\delta _{2}''=\sup_{A}\sum_{j=1}^{n}\E  | f(X)D _{j}f(X^{A})^{3} | $ from \eqref{eq:abstract-intermed-bound}, which can give a better convergence rate or less restrictive hypotheses, but it requires a specific analysis and we do not develop it below.


\medskip
 
\emph{Nearest neighbours statistics.} Let $k\geq 1,i\geq 1$,  let $\psi :(\mathbb{R}^{d})^{k}\to \mathbb{R}$ be a measurable function and let 
\begin{align*}
f(x_{1},\dots ,x_{n}):=\frac{1}{\sqrt{ n}}\sum_{i=1}^{n}\psi (x_{i}^{(1)},\dots ,x_{i}^{(k)})
\end{align*} where the $x_{i}^{(j)}$ are the $k$ nearest neighbours of $x_{i}$ among $(x_{1},\dots ,x_{n})$ for the Euclidean distance, ordered by increasing distance to $x_{i}$, with an arbitrary tie breaking rule.
Given $n$ i.i.d random variables $X_{1},\dots ,X_{n}$ in $\mathbb{R}^{d}$, in \cite{Cha08} Chatterjee obtains estimates on the Wasserstein distance between $f(X)$ and the normal law under the assumptions that for $i\neq j$, $\|X_{i}-X_{j}\|$ is a continuous random variable. He obtains the bounds, for $p\geq 8$, 
\begin{align*}
\delta _{1}\leqslant C_{d}\frac{k^{4}\gamma _{p}^{2 }}{\sigma ^{2}n^{(p-8)/2p}},\\
\delta _{2}\leqslant C_{d}\frac{k^{3}\gamma _{p}^{3 }}{\sigma ^{3}n^{(p-6)/2p}},
\end{align*}
where $\gamma _{p}:= \left( \E | \psi (X_{1},\dots ,X_{n}) |^{p} \right)^{1/p}, C_{d}>0 $.
These bounds are obtained through \cite[Theorem 2.5]{Cha08}, which is similar to   Theorem \ref{thm:dependancy}, where our bound on $\delta _{1}'$ is already smaller or equal to the bound on $\delta _{1}$ from \cite[Theorem 2.5]{Cha08}, up to a constant, see Remark \ref{r:zezz}. Therefore we have $\delta _{1}'\leqslant C \delta _{1}$. In order to obtain an explicit bound on the Kolmogorov distance, it therefore only remains to bound $\delta _{2}'$.
 In \cite{Cha08} it is shown that $ \E \sup_{j=1}^{n} | \Delta _{j}f(X) | ^{p}\leqslant (n^{2}+n)n^{-p/2}\gamma _{p}^{p}  $
 from where the bounds 
\begin{align*}
\delta _{1}' &\leqslant C_{k,d}n^{1/2}\left( \E  \sup_{j=1}^{n} | \Delta _{j}f(X) | ^{p} \right)^{2/p}\leqslant C_{k,d}n^{4/p}n^{1/2}n^{-1}\gamma _{p}^{2}=C_{k,d}\frac{\gamma _{p}^{2}}{n^{( p-8)/2p}}\\
\delta _{2}&\leqslant C_{k,d}n\left( \E  \sup_{j=1}^{n} | \Delta _{j}f(X) | ^{p} \right)^{3/p}\leqslant C_{k,d}\frac{\gamma _{p}^{3}}{n^{1/2-6/p}} \\
\delta _{2}'&\leqslant C_{k,d}n\left( \E \sup_{j=1}^{n} | \Delta _{j}f(X) |^{p}  \right)^{3/p}\leqslant \delta _{2}.
\end{align*}easily follow. {We observe that in \cite{Cha08} a more general situation is actually considered : for each $i$, a different functional $\psi _{i}$ is applied to $(x_{i}^{(1)},\dots ,x_{i}^{(k)})$ in the definition of $f$. However, all the explicit examples developed in such reference are purely geometric, in the sense that this subtlety is not exploited, and the functional $f (X)$ is symmetric. These examples includes the  \emph{average distance to the nearest neighbour}, the \emph{degree count in the nearest-neighbour graph,} and the\emph{Levina-Bickel statistic with parameter $k$}, which is defined by 
\begin{align*}
f(x_{1},\dots ,x_{l})=\frac{1}{n}\sum_{i=1}^{n}\left( \frac{1}{k-1}\sum_{j=1}^{k-1}\log\left( \frac{\|x_{i}-x_{i}^{(k)}\|}{\|x_{i}-x_{i}^{(j)}\|} \right) \right).
\end{align*} }

\medskip   
   
\noindent \textit{Flux through a random conductor}. In \cite{Nol}, Nolen considers the solution of an elliptic partial differential equation with a stationary random conductivity coefficient $a(x)$ over the torus $[0,L)^{d}, L>0.$ The random function $a(x)$ depends on the local contributions of a set of i.i.d variables $Z=(Z_{1},\dots ,Z_{k})$ indexed by $\mathbb{Z} ^{d}\cap [0,L)^{d}$. He derives a bound on the Wasserstein distance between the normal law and the average flux $\Gamma (Z)$ of the solution. He obtains the bounds 
\begin{align}
\label{eq:bound-wass-flux}
 \delta _{1}&\leqslant CL^{-3d/2}\sigma ^{-2}\log(L)\left( \E \Phi _{0}^{8q} \right)^{1/2q},\\
\delta_{2}&\leqslant C\sigma ^{-3}L^{-2d}\E \Phi _{0}^{6},
\end{align}where $\sigma ^{2}$ is the variance and $\Phi _{0}$ is an integral related to the gradient of the solution over $[0,1)^{d}$ (see \cite{Nol} for details). 

{ Our method allows one to extend this result to the Kolmogorov distance, under slightly stronger assumptions. Gloria and Nolen \cite{gn} have also used Theorem 4.2 for a Kolmogorov Berry-Essen bound with a discretised version of the problem.}
Once again, the simple inequality $ |  | a | - | b |  | \leqslant  | a-b | ,a,b\in \mathbb{R},$
yields that the upper bound on $\var(T(Z,Z') | Z')$ derived in \cite[(2.25)-(2.27)]{Nol} and then used in (4.53) can be used in an exact similar fashion to bound $\var(T'(Z,Z') | Z')$ where $T'$ is defined as in our Theorem \ref{thm:SchulteSchatterjeebound}. This yields that $\delta _{1}'$ satisfies the same bound as $\delta _{1}$, up to a constant. Then, \cite[Lemma 4.1]{Nol} provides the estimate 
\begin{align*}
\E  | \Delta _{j}\Gamma (Z) | ^{q}\leqslant C_{q}L^{-qd}\E  | \Phi _{0}(Z) | ^{2q}
\end{align*}
 which readily yields the first term of \ref{eq:bound-wass-flux}, and the bound on the Kolmogorov distance
\begin{align*}
\delta _{1}+\delta _{2}+\delta _{1}'+\delta _{2}'\leqslant C(\delta _{1}+L^{-2d}\sqrt{\E  | \Phi _{0} | ^{12}}).
\end{align*}Note that the new condition $\E  | \Phi _{0} | ^{12}<\infty $ might be weakened if one uses (\ref{eq:abstract-intermed-bound}) instead of (\ref{eq:abstract-bound}), as it is done  in the proof of Theorem \ref{thm:voronoi-tcl}.\\

\bibliographystyle{plain}

\bibliography{berryessen.bib}
\end{document}